\documentclass[10pt]{article}  
\usepackage{amssymb}              
\usepackage{amsthm}
\usepackage{eucal}
\usepackage{verbatim}
\usepackage[dvips]{graphicx}
\usepackage{multirow}
\usepackage{fancyhdr}
\usepackage{enumerate}
\usepackage{amsmath}	
\newtheorem{theorem}{Theorem}
\newtheorem{lemma}{Lemma}
\newtheorem{definition}{Definition}
\newtheorem{remark}{Remark}
\newtheorem{proposition}{Proposition}

\makeatletter
\newcommand{\leqnomode}{\tagsleft@true}
\newcommand{\reqnomode}{\tagsleft@false}
\makeatother

\newcommand{\Id}{\mathrm{I}}

\begin{document}

\title{Hydrodynamic Limit for a Fokker-Planck Equation with Coefficients in Sobolev Spaces.}   
\author{Ioannis Markou}         
\date{June 2016}    
\maketitle

\begin{abstract}
  In this paper we study the hydrodynamic (small mass approximation) limit of a Fokker-Planck equation. This
  equation arises in the kinetic description of the evolution of a particle system immersed in a
  viscous Stokes flow. We discuss two different methods of hydrodynamic convergence. The first method
  works with initial data in a weighted $L^{2}$ space and uses weak convergence and the extraction of convergent
  subsequences. The second uses entropic initial data and gives an $L^{1}$ convergence to the solution of
  the limit problem via the study of the relative entropy.
\end{abstract}


\textbf{2000 MR Subject Classification}: 35Q35, 35Q84.

\section{Introduction}

\subsection{Introduction to the problem}

We study the hydrodynamic limit for a Fokker-Planck equation that
arises in the modeling of a system of large particles
immersed in a much larger number of micromolecules. Examples of such
particle systems include dilute solutions of polymers that arise often in industrial settings
\cite{BACH,Do,DoEd,JaCl,Ki}. Typically, macromolecules (or more precisely, the monomer parts
they are comprised of) are modeled by ideal spheres
whose interactions are mediated by interactions with the micromolecules. We model the micromolecules as
an incompressible fluid governed by Stokes flow. The interactions of these idealized particles with the fluid are
modeled by admissible boundary conditions, Brownian noise and the introduction of damping.

The dynamics of particle motion is described by a phase-space vector $(x,v)\in
\mathbb{R}^{n}_{x}\times \mathbb{R}^{n}_{v}$.
If the statistics of the particle motion is described
by the probability density $f(t,x,v)\geq 0$, then
the evolution of $f$  is governed by the
Fokker-Planck equation
\begin{equation} \label{F-P*} \partial_{t}f +v \cdot \nabla_{x}f+\frac{1}{m}
\nabla_{v}\cdot (F f)=\frac{1}{m^{2}}
\nabla_{v}\cdot (G(x)\nabla_{v}f) ,\end{equation}
with the particle mass represented by $m$. The force $F(t,x,v)$ in
our model is chosen so that $F(t,x,v)=-\nabla V(x)-G(x)v$, where $V(x)$ is a
potential that depends on the particles' configuration and $G(x)v$ is the damping (hydrodynamic) force term.
The potential $V(x)$ captures all interactions between beads that are not mediated by the fluid. This allows,
particularly, for the incorporation in the model of any type of spring forces between beads.

The Fokker-Planck equation \ref{F-P*} is naturally
associated to the phase-space stochastic differential system
\begin{align*} \left\{ \begin{array}{ll}  &\dot{x}(t)=v , \\
&\dot{v}(t)=-\frac{1}{m}(G(x)v+\nabla V(x))
+\frac{\sqrt{2}}{m}G^{1/2}(x)\dot{W}(t),  \end{array} \right. \end{align*}
where $W(t)$ is the centered Gaussian vector in $\mathbb{R}^{n}$ with covariance
$\mathbb{E}( \dot{W}(t)\dot{W}(t')) \break =\Id \delta(t-t')$.
Here $\mathbb{E}$ stands for expectation with respect to Gaussian measure. The inclusion of the friction matrix $G(x)$
in the Brownian forcing term is an instance of the fluctuation-dissipation theorem which asserts
that fluctuations caused by white noise and the response to small perturbations applied to the system are in balance.
This is evident by the Einstein-Smoluchowski relation \cite{Ein,Sm} that states that the diffusion tensor (related to
thermal motion) is proportional to friction $G(x)$.

Equation \ref{F-P*} is very important in the description of
polymer models when inertial effects are involved. This
is reminiscent of the inertial kinetic models in the
work by P.Degond and H.Liu \cite{DeLi}. Therein, the authors introduce
novel kinetic models for Dumbell-like and rigid-rod polymers in the presence of
inertial forces and show formally that when inertial effects vanish the limit is
consistent with well accepted macroscopic models in polymer rheology. A direct quote from \cite{DeLi}
reasons on the importance of kinetic models involving inertial effects in describing polymer sedimentation:
``In current kinetic theory models for polymers, the
inertia of molecules is often neglected. However, neglect of inertia in some cases
leads to incorrect predictions of the behavior of polymers.
The forgoing considerations indicate that the inertial effects are of
importance in practical applications, e.g., for short time characteristics
of materials based on the relevant underlying
phenomena''.

One of the differences with the theory in the Degond \& Liu work is that we
take into account hydrodynamic interactions between $N$ particles
with the use of the symmetric, non negative $3N \times 3N$
friction tensor $G(x)$, that contains all the information
for these interactions. These hydrodynamic
interactions are the result of a particle's motion that
perturbs the fluid and has an effect on other particles' movement. The constant friction case
$G(x)=\gamma \Id$ (for $\gamma >0$) is interesting in its own
right as it corresponds to particles that ``sink
freely'' without any hydrodynamic type of interaction
between them. In this trivial case, there is no account of hydrodynamic effects
and the parabolic limit is derived with no difficulty as we show.
In a similar spirit as in \cite{DeLi},
our goal is to show rigorously that
equation \ref{F-P*} leads to the derivation of a well accepted Smoluchowski type
of equation when inertial effects are ignored (see Theorem 1).

Before we proceed with the details of the limiting
approximation, we should note the difficulties
in computing the exact formula for friction $G(x)$ (or most
commonly the mobility $\mu=G^{-1}(x)$) for every $N$
particle configuration. In practice, this would involve
solving a linear Stokes system with very complicated
boundary conditions, i.e. the $N$ particles' surface. A particular modelling
problem is the appropriate
way to compute these interactions for overlapping particles and particles that are almost touching. More
specific, for particles that are very close, integrable singularities
of the friction tensor are possible (lubrication effects).
Below we give the two most important approximations
of the mobility tensor used in simulations.

The first non-trivial approximation to mobility
is the Oseen tensor that corresponds to Green's kernel
solution of a Stokes problem for point particles \cite{DoEd,KiKa}. For $N$ particles with
centers $\{x_{i}\}_{i=1}^{N}$, radius $a$, in
a fluid with viscosity $\eta$, the Oseen tensor
$\mu^{OS}=[\mu^{OS}_{ij}]_{i,j=1}^{N}$ is a $3N \times 3N$ tensor with $3 \times 3$ blocks
\begin{equation*} \mu_{ij}^{OS}=\left \{ \begin{array}{ll} \frac{1}{8 \pi \eta |R_{ij}|}
\left(\Id+\hat{R}_{ij}\otimes \hat{R}_{ij} \right),
&  i\neq j \\ \frac{1}{6 \pi \eta a}\Id, &  i=j ,
\end{array} \right. \end{equation*}
where $R_{ij}=x_{i}-x_{j}$ and $\hat{R}_{ij}=R_{ij}/|R_{ij}|$.
This approximation works
quite well when particles are well separated ($|R_{ij}|\gg a$), but
it is degenerate for particle
configurations that involve particles relatively close. This
implies that the Oseen tensor cannot be a meaningful choice that leads
to a well-posed theory (in the sense of existence, uniqueness, macroscopic limit, \ldots).

The Rotne-Prager-Yamakawa approximation of the
mobility tensor \cite{RoPra,Ya} is a non negative correction
to the Oseen tensor that applies to all particle configurations. In addition, Rotne
and Prager \cite{RoPra} obtained a way to calculate
mobilities for overlapping spheres. The expression for the RPY $3N \times 3N$ mobility has blocks
\begin{equation*} \mu_{ij}^{RPY}=\left \{ \begin{array}{ll} \frac{1}{8 \pi \eta |R_{ij}|}
\left[ \left(1+\frac{2a^2}{3|R_{ij}|^2}\right)\Id +\left(1-\frac{2a^2}{|R_{ij}|^2}\right)
\hat{R}_{ij} \otimes \hat{R}_{ij}\right], & |R_{ij}|>2a \\ \frac{1}{6 \pi \eta a}
\left[ \left(1-\frac{9|R_{ij}|}{32a}\right)\Id +\frac{3|R_{ij}|}
{32a} \hat{R}_{ij}\otimes \hat{R}_{ij}\right], & |R_{ij}| \leq 2a .
\end{array} \right. \end{equation*}
Eigenvalues of the tensor depend continuously on the particles'
positions, they are bounded and the RPY mobility is
locally integrable in space. On the
other hand, the tensor is still not strictly positive. In more detail, when
two spheres (of radius $a$) almost coincide
and their centers have distance $d=|R_{ij}|\ll a$, then the minimum
eigenvalues $\lambda_{min}(x_{1},x_{2})$ of RPY
are of order $O(d)$. This in turn implies that the
friction associated to the RPY tensor is of order $O(\frac{1}{d})$
and hence gives an integrable singularity (satisfying the assumptions of Theorem 1).

We should note that the exact computation of the eigenvalues of the RPY mobility for
$N>2$ is impossible and the problem of directly obtaining the best lower bounds for
$\lambda_{min}(x_{1},\ldots,x_{N})$ is still open. On the other
hand, the additive nature of hydrodynamic interactions suggests
a bound from below that is linear with respect to particle distances. For instance, for
$N$ particles in a configuration with all interparticle distances equal to $d=|R_{ij}| \ll a \quad \forall i,j$, the
minimum eigenvalues can be computed exactly and are once again of order $O(d)$.
Moreover, for two nearly touching spheres with dimensionless
gap parameter $\xi=\frac{|R_{ij}|}{a}-2$ (with $\xi \ll 1$),
lubrication theory suggests that the leading order of the friction tensor is $O(\frac{1}{\xi})$
\cite{CiFeHiWaBl, JeOn, KiKa, Re}. The last observation implies that any
physically meaningful, non negative
choice of mobility should satisfy the assumptions of the first result.

In this work, we identify the conditions on the hydrodynamic mobility
so that a derivation of the macroscopic limit is possible.
Let us only mention that the particle system described here, without
the inclusion of Brownian motion, is not new
in math literature (see e.g. \cite{JaOt,JaPe}). For example, \cite{JaOt} gives a
study of the dynamics of particle motion
when the inertia of particles is neglected so that
the regime in which particles sink
approximately with no hydrodynamic interactions is established.

In order to study the diffusion limit of kinetic equation \ref{F-P*}, we need to introduce the
appropriate scaling to separate conservative and dissipative terms. We repeat
the scaling procedure in \cite{DeLi} that involves the
change of variables,
\begin{equation*} m=\epsilon^{2}, \qquad v'=\epsilon v, \qquad x'=x .\end{equation*}
Thus, \ref{F-P*} becomes
(after we re-introduce the notation for $x$, $v$
in the place of $x'$, $v'$ and set initial conditions) the Cauchy problem

\begin{equation} \label{F-P} \begin{gathered} \partial_{t}f_{\epsilon}+L_{\epsilon}f_{\epsilon}=0,\qquad
f_{\epsilon}(0,x,v)=f_{0,\epsilon}(x,v), \\ \text{with} \quad
 L_{\epsilon}=\frac{1}{\epsilon}
\left( v \cdot \nabla_{x}f_{\epsilon}-\nabla V(x) \cdot \nabla_{v}f_{\epsilon}\right)
-\frac{1}{\epsilon^{2}}\nabla_{v}\cdot \left( G(x)(\nabla_{v}f_{\epsilon}
+vf_{\epsilon})\right).\end{gathered} \end{equation}

Our main objective is to study the (zero mass) limit $\epsilon \to 0$, for both
$f_{\epsilon}$ and the hydrodynamical density $\rho_{\epsilon}:=\int f_{\epsilon}\, dv$, with integration
assumed everywhere over $\mathbb{R}^{n}_{v}$ ($\mathbb{R}^{n}_{x} \times \mathbb{R}^{n}_{v}$
when spatial variables are also involved). The second term of $L_{\epsilon}$ in
\eqref{F-P} is responsible for the system approaching local equilibrium Gibbs states $\rho \mathcal{M}(v)$,
with $\mathcal{M}(v)$ being the standard Maxwellian distribution
\begin{equation} \label{GD} \mathcal{M}(v)=
e^{-\frac{|v|^{2}}{2}}/(2 \pi)^{\frac{n}{2}} \end{equation}
and $\rho$ the limit of $\rho_{\epsilon}$.
The hydrodynamic limit study is the PDE analog of the Kramers-Smoluchowski
approximation for the Stochastic DE that corresponds to \ref{F-P*} \cite{Fr}.

Similar macroscopic limits in the parabolic scaling regime have been considered by many authors
in the past, and for various collision operators that lie in the fast scale $\epsilon^{-2}$.
A discussion of the literature cannot, by any means, be inclusive. We only outline here some
works that are relevant \cite{BGPS,DeGoPo,DoMaOeSc,GoPo,Po,PoSc}. For instance, in \cite{DeGoPo}
this limit is considered for the linear Boltzmann equation with a collision operator of the form
$\int \sigma (x,v,\omega)f(\omega)\, d\mu(\omega)-f \int \sigma(x,v,\omega) \, d\mu(\omega)$,
for a $\sigma$-finite measure $d \mu(\omega)$, and under the assumption that there exists
a unique stationary state $F(x,v)$ for which
$$F(x,v)\int \sigma (x,v,\omega)\, d\mu (\omega)=\int
\sigma(x,v,\omega)F(x,\omega)\, d\mu(\omega) \quad \text{a.e.}$$  The collision kernel $\sigma(x,v,\omega)$
is assumed measurable with $\int \sigma(x,v,\omega) \, d\mu(\omega)<\infty$ and it does not
satisfy the micro-reversibility condition
$\sigma(x,v,\omega)F(x,\omega)=\sigma(x,\omega,v)F(x,v)$. Such models are prominent
in the theory of plasmas, semiconductors, rarefied gases
etc. In \cite{PoSc} the authors study the parabolic limit
for the nonlinear Boltzmann operator $\mathcal{M}(v)(1-f)\int f \, dv -f
\int \mathcal{M}(v)(1-f)\, dv$. This operator appears in the study of semiconductors, where
$f(t,x,v)$ is the fraction of occupied states (occupancy number). The operator
leads to relaxation to the Fermi-Dirac distribution
$f_{F-D}(\mu,v)=\left( 1+e^{(\frac{1}{2}|v|^{2}-\mu)}\right)^{-1}$, where $\mu$ is
the Fermi energy that depends implicitly on $\rho(\mu)=\int f_{F-D}(\mu ,v) \, dv$. When the limit
is considered for $\epsilon=\frac{\tau}{L}\to 0$ (mean free path $\tau$ is
small compared to characteristic length scale $L$) then $f \to f_{F-D}$ and the Fermi energy
$\mu$ satisfies the diffusive equation $\partial_{t}\rho(\mu)=\nabla_{x}\cdot(D(\mu)\nabla_{x}(\mu-V))$,
for a diffusive coefficient $D(\mu)$ with an explicit structure. The electrostatic potential
$V(t,x)$ appears in the transport term $v \cdot \nabla_{x}f-\nabla_{x}V \cdot \nabla_{v}f$, which is in scale $\epsilon^{-1}$.
The Rosseland approximation for the radiative transfer equation
has been studied in \cite{BGPS}. Equations that lead to nonlinear
diffusions in the limit have been considered in \cite{DoMaOeSc}.

In this article, the derivation of a
convection-diffusion limit is carried out for a linear Fokker-Planck equation with dominating friction and Brownian
forcing terms governed by an anisotropic tensor $G(x)$. The equation is of particular importance in the theory of
particles moving in Stokes flows. The limiting Smoluchowski equation that we derive is the cornerstone
of the kinetic theory of polymer chains in dilute solutions \cite{Do,DoEd,Ki}.

\subsection{Main Theorems}

We now bring our attention to the two main results of hydrodynamic convergence.
In both of the results we are about to present,
we assume that the solution to equation \ref{F-P}
is weak (in the sense that will be explained in Section 2) thus allowing for quite irregular
coefficients. We make two assumptions. First, we assume a non
degenerate friction $G(x)$ such that $G^{-1}(x)$ exists a.e and second that
$e^{-V(x)}\in L^{1}(\mathbb{R}^{n}_{x})$.
These assumptions suggest that there exists a unique
global equilibrium state explicitly given by
\begin{equation} \label{Equil} \mathcal{M}_{eq}(x,v)=e^{-V(x)}\mathcal{M}(v)/Z ,
\quad \text{with} \quad Z=(2 \pi)^{\frac{n}{2}}\int e^{-V(x)}\, dx. \end{equation}
We also consider $V(x)$ bounded
from below in the sense that $\inf \, V(x)>-\infty$.

In the first theorem, we establish weak convergence of the hydrodynamic
variable $\rho_{\epsilon}(t,x)$ based on weak compactness arguments.
The proof is actually quite straightforward. We assume
a solution of \ref{F-P} in the mild-weak sense. Such a solution $f_{\epsilon}$ lives in $C(\mathbb{R}_{+},\mathcal{D}'(\mathbb{R}^{n}_{x} \times
\mathbb{R}^{n}_{v}))$. We
also make the assumption that the initial data are
in the weighted $L^{2}_{\mathcal{M}_{eq}}$ space, where
$L^{2}_{\mathcal{M}_{eq}}=\mathcal{M}_{eq}L^{2}
(\mathcal{M}_{eq}\, dv \, dx)=\mathcal{M}_{eq}L^{2}(d \mu)$ (for a
measure $\mu$ with density $\mathcal{M}_{eq}$) i.e.
\begin{equation} \label{ApEs1} \| f_{\epsilon}(0,x,v)\|_{L^{2}_{\mathcal{M}_{eq}}} < C
\quad \forall \epsilon >0, \, \text{for some} \quad C>0 . \end{equation}
We prove in Section 3 the following theorem.
\begin{theorem}
Let $f_{\epsilon}$ be a mild-weak solution to \ref{F-P} with
bounded initial energy $\|f_{\epsilon}(0,\cdot,\cdot)
\|_{L^{2}_{\mathcal{M}_{eq}}}< \infty $ (uniformly in
$\epsilon>0$), and let $\rho_{\epsilon}$ be the hydrodynamical
density $\rho_{\epsilon}:=\int f_{\epsilon}\, dv$.
Assume that the non-degenerate a.e. friction tensor $G(x)$ and potential $V(x)$ satisfy conditions :
$G^{-1}(x) \, \& \, G(x) \in (L^{1}_{loc}(\mathbb{R}^{n}_{x}))^{n\times n}$, \, $\nabla V(x)\in (L^{2}_{loc}(\mathbb{R}^{n}_{x}))^{n}$, \, $G^{-1/2}\nabla V(x)\in (L^{2}_{loc}(\mathbb{R}^{n}_{x}))^{n}$
and $e^{-V(x)}\in L^{1}(\mathbb{R}^{n}_{x})$. In the limit $\epsilon \to 0$,
we have the following convergence
\begin{equation*} \rho_{\epsilon} \rightharpoonup \rho \qquad \text{in}
\quad C([0,T],w-L^{2}(dx)), \end{equation*}
where $\rho$ is the solution to the Smoluchowski equation
\begin{equation} \label{Sm} \partial_{t}\rho=\nabla_{x}\cdot (G^{-1}(\nabla_{x}\rho+ \nabla V(x)\rho))
\qquad \text{in} \quad C([0,T],\mathcal{D}'(\mathbb{R}^{n}_{x})).\end{equation}
\end{theorem}

In the second theorem, we use the relative entropy functional to
prove an $L^{1}$ convergence result. The relative entropy $H(f|g)$ between two densities
$f,g$ is defined by
\begin{equation} \label{RE} H(f|g)=\iint f \log \frac{f}{g}\, dv \, dx ,\end{equation}
and in the present work it will be used to control the
distance of a solution $f_{\epsilon}$ of \ref{F-P} from
the local Gibbs state $\rho \mathcal{M}(v)$ as $\epsilon \to 0$.

The relative entropy has been used in the study of many asymptotic problems.
The earliest example appears to be in the study of the
hydrodynamic limit for the Ginzburg-Landau problem in \cite{Yau}.
In \cite{Va}, the author takes a probabilistic approach to the use of relative entropy.
Other more elaborate cases include the Vlasov-Navier-Stokes
system \cite{GoJaVa2}, hydrodynamic limits for the Boltzmann equation \cite{GoLeRa}.

To prove Theorem 2, we make the following assumptions. First, we
need conditions that give control of the hydrodynamical tensor $G^{-1}(x)$ and potential $V(x)$, i.e.
\begin{align}
\left.
\begin{array}{c}
  G^{-1}(x) \geq \lambda I \quad \text{for some} \quad \lambda >0 \quad \text{and} \\
  \|\nabla^{k}G^{-1}\|_{L^{\infty}(\mathbb{R}_{x}^{n})}<\infty, \quad
 \|\nabla^{k}(G^{-1}\nabla V(x))\|_{L^{\infty}(\mathbb{R}_{x}^{n})}<\infty, \quad  1 \leq k \leq 3.
\end{array}\right\} \tag{A1}\label{A1}
\end{align}
We also assume that the initial condition $\rho(0,x)$ to equation \ref{Sm} satisfies
\begin{align}
\left.
\begin{array}{c}
  a e^{-V(x)}\leq \rho(0,x) \leq A e^{-V(x)} \quad \text{for some} \quad A>a>0 \quad \text{and} \\
  \rho(0,x) / e^{-V(x)} \in W^{3,\infty}(\mathbb{R}^{n}_{x}).
\end{array}\right\} \tag{A2}\label{A2}
\end{align}
Finally, the use of the maximum principle for the parabolic equation \ref{Sm} in $\mathbb{R}^{n}_{x}$ requires
certain admissibility conditions at
infinity. We can choose for instance the following condition for a given $T>0$,
\begin{equation} \sup_{0 \leq t \leq T} \limsup_{x \to \infty} \Big| \nabla^{k}
\frac{\rho(t,x)}{e^{-V(x)}}\Big|\leq C_{k} \quad \text{for} \quad C_{k}>0, \quad 0\leq k \leq 3, \tag{A3} \label{A3} \end{equation}
where $|\cdot|$ is the Hilbert-Schmidt norm of the tensor. In Section $4$ we prove
\begin{theorem}
Let $f_{\epsilon}(0,x,v)$ be initial data to the F-P
equation \ref{F-P} such that $f_{\epsilon}(0,x,v) \geq 0$,
satisfying the energy bound
\begin{equation} \label{ApEs2} \sup \limits_{\epsilon >0} \iint f_{\epsilon}(0,x,v)(1+V(x)+
|v|^{2}+\log f_{\epsilon}(0,x,v))\, dv \, dx < C <\infty .\end{equation}
Moreover, we assume that $e^{-V(x)}\in L^{1}(\mathbb{R}^{n}_{x})$ and that the hydrodynamic tensor
$G^{-1}(x)$ and potential $V(x)$ satisfy condition \ref{A1}.
Let $\rho(0,x)\in \mathcal{D}'(\mathbb{R}^{n}_{x})$ be initial data to the limit equation \ref{Sm},
satisfying \begin{equation*} \int \rho(0,x)\, dx=\iint f_{\epsilon}(0,x,v)\, dv \, dx =1, \end{equation*}
as well as condition \ref{A2}.
We finally make the assumption that the initial data are prepared so that
\begin{equation*} H(f_{\epsilon}(0,\cdot,\cdot)|\rho(0,\cdot)\mathcal{M}(v))\to 0 \quad \text{as}
\quad \epsilon \to 0 .\end{equation*}
Then, for any $T>0$, if $\rho(t,x)\in C([0,T],\mathcal{D}'(\mathbb{R}^{n}_{x}))$ is a solution to the limit
equation that satisfies \ref{A3}, we have
\begin{equation*} \sup \limits_{0 \leq t \leq T}H(f_{\epsilon}(t,\cdot,\cdot)|\rho(t,\cdot)\mathcal{M}(v))
\to 0 \qquad \text{as}\quad \epsilon \to 0 .\end{equation*}
\end{theorem}

The rest of the paper is organized as follows. In the next section, we give a formal
derivation of the macroscopic limit and present the main steps in the proof of the two theorems
mentioned above. We also give an exact description of the type of solutions we assume for problem \ref{F-P},
in each theorem. Sections 3 \& 4 are devoted to the proof of each
theorem with all the a priori estimates.

\section{Formal derivation of the limit problem and outline of proofs of the Main Theorems}

We begin by writing the collision operator in form
\begin{equation*} \nabla_{v} \cdot (G(x)(\nabla_{v}f_{\epsilon}+vf_{\epsilon}))
=\nabla_{v}\cdot \left(\mathcal{M}(v)G(x)\nabla_{v}
\left( \frac{f_{\epsilon}}{\mathcal{M}(v)}\right)\right).\end{equation*}
This form is indicative of why the collision part of $L_{\epsilon}$
is responsible for the dissipation of energies.
Let us now introduce the hydrodynamical variables for the density $\rho_{\epsilon}$, the
flux vector $J_{\epsilon}$, and the kinetic pressure tensor $\mathbb{P}_{\epsilon}$ of the
particle system, i.e.
\begin{equation} \label{HV} \rho_{\epsilon}(t,x):=\int f_{\epsilon}\, dv,\quad J_{\epsilon}(t,x):=\int
v f_{\epsilon}\, dv, \quad \mathbb{P}_{\epsilon}(t,x):=\int v \otimes v f_{\epsilon}\, dv .\end{equation}
In the study of the limit $\epsilon \to 0$, we want to derive an equation for the hydrodynamic
variable $\rho(t,x)$ which is formally the limit of $\rho_{\epsilon}$.

First, integrating \ref{F-P} in velocity space, we obtain
\begin{equation} \label{Eq-Rhoe} \partial_{t}\rho_{\epsilon}+\frac{1}
{\epsilon}\nabla_{x}\cdot J_{\epsilon}=0 .\end{equation}
We want to derive an expression for the evolution of $J_{\epsilon}$ and
study the order of magnitude in $\epsilon$ of the terms involved in it. In the
derivation of the equation for the first moment, we multiply
the F-P eq. \ref{F-P} by $v$ and integrate in velocity.
The resulting equation is
\begin{equation} \label{Eq-Je} \epsilon^{2} \partial_{t}J_{\epsilon}(t,x)+\epsilon (\nabla_{x}\cdot
\mathbb{P}_{\epsilon}(t,x)+ \nabla V(x)\rho_{\epsilon}(t,x))=-G(x)J_{\epsilon}(t,x).\end{equation}
As we show in our proof, the main contributions in \ref{Eq-Je}
come from the rhs term and the second and third terms in the lhs. Indeed, rewriting the
pressure tensor we have
\begin{equation*} \int v_{i}v_{j}f_{\epsilon}\, dv=-\int \partial_{v_{i}}(\mathcal{M})
v_{j}\frac{f_{\epsilon}}{\mathcal{M}}\, dv=\int \delta_{ij}f_{\epsilon}\, dv+ \int
\mathcal{M}\partial_{v_{i}}\left( \frac{f_{\epsilon}}{\mathcal{M}}\right)v_{j}\, dv ,\end{equation*}
which implies
\begin{equation} \label{Eq-Pe}  \mathbb{P}_{\epsilon}(t,x)= \rho_{\epsilon}I+
\int \mathcal{M} \nabla_{v}\left( \frac{f_{\epsilon}}
{\mathcal{M}} \right)\otimes v \, dv .\end{equation}
With the help of \ref{Eq-Pe}, equation \ref{Eq-Je} now gives
\begin{align} \nonumber J_{\epsilon}=-\epsilon G^{-1}(x)(\nabla_{x}\rho_{\epsilon}
+\nabla V(x)\rho_{\epsilon})-&\epsilon^{2}G^{-1}(x)\partial_{t}J_{\epsilon} \\ \label{Eq-Je2}
-&\epsilon G^{-1}(x)\nabla_{x}\cdot \int \mathcal{M} \nabla_{v}
\left( \frac{f_{\epsilon}}{\mathcal{M}}\right)\otimes v \, dv .\end{align}
The last term in \ref{Eq-Je2} contains the part
$\int \mathcal{M}\nabla_{v}\left( \frac{f_{\epsilon}}{\mathcal{M}}\right)\otimes v\, dv$ which
appears in the expression for $\mathbb{P}_{\epsilon}(t,x)$. This term will be shown to be of order $\epsilon$
if one uses the appropriate a priori estimate e.g. in $L^{2}(\mu)$. This implies that
in the limit $\epsilon \to 0$, we should be able to establish that
$\mathbb{P}_{\epsilon}(t,x)\to \rho(t,x) I$. The term
$\epsilon^{2}G^{-1}(x)\partial_{t}J_{\epsilon}$ will be shown to be of order $\epsilon^{2}$,
as long as we give an appropriate interpretation to a solution $J_{\epsilon}(t,x)$ of \ref{Eq-Je2}.
Hence, we will justify rigorously the following expansion for $J_{\epsilon}$,
\begin{equation} \label{FEJ} J_{\epsilon}(t,x)=-\epsilon G^{-1}(x)(\nabla_{x}\rho_{\epsilon}
+\nabla V(x)\rho_{\epsilon}) + \epsilon^{2}\ldots .\end{equation}
Finally, as we let $\epsilon \to 0$, the system of equations \ref{Eq-Rhoe} \& \ref{FEJ} converges to
\begin{align*} & \qquad \partial_{t}\rho +\nabla_{x}\cdot J =0 \\
J&=-G^{-1}(x)(\nabla_{x}\rho +\nabla V(x)\rho) , \end{align*}
where $J$ is the limit of $J_{\epsilon}/\epsilon$. At the same time, since $f_{\epsilon}$
approaches local Gibbs states, it follows that $f_{\epsilon}\to \rho(t,x)\mathcal{M}(v)$.
All this is enough to suggest that the limit
equation for $\rho$ solves the Smoluchowski equation
\ref{Sm}.

It is now time to give a brief step by step
outline of the proof of Theorems 1 \& 2. We begin with the first result,
in which we show weak convergence to the solution of the limiting problem.

In the first step of the proof, we decompose $f_{\epsilon}(t,x,v)$ into a local equilibrium state
$\mathcal{M}(v)\rho_{\epsilon}(t,x)$, and a deviation $\mathcal{M}(v)\tilde{g}_{\epsilon}(t,x,v)$. With
the help of the a priori energy estimate we can extract convergent subsequences for $\rho_{\epsilon}(t,x)$,
$\tilde{g}_{\epsilon}$, and $\frac{1}{\epsilon}G^{1/2}(x)\nabla_{v}\tilde{g}_{\epsilon}(t,x,v)$. Then,
we can show that $\rho_{\epsilon}$ is compact in $C([0,T],w-L^{2}(\mathbb{R}^{n}_{x}))$,
for any $T>0$. Next, we write an evolution equation for $\tilde{g}_{\epsilon}(t,x,v)$ (an equation
in the distributional sense) and pass to the limit $\epsilon \to 0$. To achieve this, since we are dealing with a
weak formulation, we have to find the order in $\epsilon$ of each integral term in this equation and
ignore all the lower order terms in $\epsilon$. The last step is to use the limit
equation for $\tilde{g}_{\epsilon}(t,x,v)$ and the limit equation for $\rho_{\epsilon}(t,x)$ to
derive the Smoluchowski equation.

In terms of the type of solutions we work with, we shall assume
that the operator $L_{\epsilon}$ generates a continuous semigroup in $L^{2}_{\mathcal{M}_{eq}}$, so we write
$f_{\epsilon}(t,x,v)=e^{-tL_{\epsilon}}f_{\epsilon}(0,x,v)$. Using the maximum principle and
energy dissipation (see Section 3), it
is easy to show that solutions to $\partial_{t}f_{\epsilon}+L_{\epsilon}f_{\epsilon}=0$ remain
bounded in $L^{2}_{\mathcal{M}_{eq}}\cap L^{\infty}$. We define
\begin{definition}
A mild-weak solution $f_{\epsilon}$ of \ref{F-P} lies in the
space \begin{equation} \label{Sp1} f_{\epsilon} \in C(\mathbb{R}_{+};\mathcal{D}'(\mathbb{R}_{x}^{n}\times
\mathbb{R}^{n}_{v}))
\cap L^{\infty}_{loc}(\mathbb{R}_{+};L^{2}_{\mathcal{M}_{eq}}\cap L^{\infty}) \end{equation}
and satisfies
\begin{align*}  \iint f_{\epsilon}(T,\cdot,\cdot)& \varphi(\cdot,\cdot) \, dv \, dx
-\iint f_{\epsilon}(0,\cdot,\cdot)\varphi(\cdot,\cdot)\, dv \, dx \\ &
-\frac{1}{\epsilon}\int_{0}^{T}\!\!\!\! \iint \left( v \cdot \nabla_{x}\varphi
-\nabla V(x)\cdot \nabla_{v}\varphi\right) f_{\epsilon} \, dv \, dx \, ds \\ & \qquad \qquad
+\frac{1}{\epsilon^{2}} \int_{0}^{T}\!\!\!\! \iint \nabla_{v}\varphi \cdot G(x)
(\nabla_{v}f_{\epsilon}+vf_{\epsilon})\, dv \, dx \, ds =0  \, ,\end{align*}
for any test function $\varphi(x,v) \in C^{1}_{c}(\mathbb{R}^{n}_{x}\times \mathbb{R}^{n}_{v})$ and $T>0$.
\end{definition}

For the second result, we use the relative entropy of $f_{\epsilon}$ with respect to local equilibrium states.
The relative entropy functional $H(f|g)$ between two probability
densities $f,g$ is a measure of distance between them. Indeed, by the celebrated
Csisz\'ar-Kullback-Pinsker inequality (\cite{Cs,Ku,Pin}) we have
\begin{equation*} \| f-g\|_{L^{1}}\leq \sqrt{2H(f|g)} .\end{equation*}
Thus, by finding $\lim \limits_{\epsilon \to 0} H(f_{\epsilon}|\rho \mathcal{M})$ we can control the square
of the $L^{1}$ distance between $f_{\epsilon}$ and $\rho \mathcal{M}$ in the limit $\epsilon \to 0$. Here
we show that the dissipation of relative entropy $H(f_{\epsilon}|\rho \mathcal{M})$ contains a
non negative part and remainder terms. It is important to show that these remainder terms vanish as
$\epsilon \to 0$. Once we show that in the limit the relative entropy is strictly dissipative, it will
be enough to consider initial data ``prepared'' in a way such that
$H(f_{\epsilon}(0,\cdot,\cdot)|\rho(0,\cdot)\mathcal{M})\to 0$ as $\epsilon \to 0$, so it follows that
$H(f_{\epsilon}(t,\cdot,\cdot)|\rho(t,\cdot)\mathcal{M})\to 0$ with $t \in [0,T]$, for any $T>0$.

We work with weak solutions of equation \ref{F-P}. Such solutions
have been shown to exist in \cite{LeBrLi2} for coefficients
that have a Sobolev type of regularity and satisfy certain growth assumptions (see Proposition 1 below).
\begin{definition}
A weak solution $f_{\epsilon}$ of \ref{F-P} belongs to the space
\begin{equation} \label{Sp2} X:=\{f_{\epsilon} | \,
f_{\epsilon} \in L^{\infty}([0,T];L^{1}\cap L^{\infty})\quad \& \quad
G^{1/2}\nabla_{v}f_{\epsilon} \in (L^{2}([0,T],L^{2}))^{n} \} , \end{equation}
for all times $T>0$ (with $f_{\epsilon}(0,\cdot,\cdot)\in L^{1}\cap L^{\infty}$). To be more
precise, a weak solution $f_{\epsilon}$ satisfies
\begin{align*} & \iint f_{\epsilon}(T,\cdot,\cdot)\varphi(T,\cdot,\cdot)\, dv \, dx
-\iint f_{\epsilon}(0,\cdot,\cdot)\varphi(0,\cdot,\cdot)\, dv \, dx -
\int_{0}^{T}\!\!\!\! \iint f_{\epsilon}\partial_{t}\varphi \, dv \, dx \, ds \\ &
-\frac{1}{\epsilon}\int_{0}^{T}\!\!\!\! \iint \left( v \cdot \nabla_{x}\varphi
-\nabla V(x)\cdot \nabla_{v}\varphi\right) f_{\epsilon} \, dv \, dx \, ds \\ & \qquad \qquad
+\frac{1}{\epsilon^{2}} \int_{0}^{T}\!\!\!\! \iint \nabla_{v}\varphi \cdot G(x)
(\nabla_{v}f_{\epsilon}+vf_{\epsilon})\, dv \, dx \, ds =0  \, ,\end{align*}
for any test function $\varphi(t,x,v)\in C^{1}((0,T);C^{1}_{c}(\mathbb{R}^{n}_{x}\times \mathbb{R}^{n}_{v}))\cap
C([0,T];C^{1}_{c}(\mathbb{R}^{n}_{x}\times \mathbb{R}^{n}_{v}))$.
\end{definition}
Notice that the
definition of a mild-weak solution (given earlier) is similar to the one
for weak solutions presented above. Main difference is that in the case of weak solutions,
the weak formulation requires that test functions
are also functions of time $t$. The existence of a unique weak
solution, for coefficients that are not smooth, is given in the following proposition borrowed from \cite{LeBrLi2}.
\begin{proposition} (see \cite{LeBrLi2}) Assume that the potential $V(x)$ and
diffusion $G^{1/2}(x)$ satisfy the following assumptions:
\begin{equation*}  (i)\quad G(x)v+\nabla V(x) \in (W^{1,1}_{loc}
(\mathbb{R}^{n}_{x}\times \mathbb{R}^{n}_{v}))^{n} \qquad
 (ii) \quad tr(G) \in L^{\infty}(\mathbb{R}^{n}_{x}) \qquad \end{equation*}
  \begin{equation*} (iii) \quad \frac{G(x)v+\nabla V(x)}
  {1+|x|+|v|} \in (L^{\infty}(\mathbb{R}^{n}_{x}\times \mathbb{R}^{n}_{v}))^{n} \end{equation*}
\begin{equation*} (iv)\quad G^{1/2}(x) \in (W^{1,2}_{loc}(\mathbb{R}^{n}_{x}))^{n \times n}\qquad (v)\quad
\frac{G^{1/2}(x)}{1+|x|} \in (L^{\infty}(\mathbb{R}^{n}_{x}))^{n \times n}. \end{equation*}
Then, given initial data $f_{\epsilon}(0,\cdot,\cdot)\in L^{1}\cap L^{\infty}$, there exists a
unique weak solution $f_{\epsilon}$
of \ref{F-P} that belongs to $X$.
\end{proposition}

\section{Diffusive limit via weak compactness. Proof of Theorem 1}

\subsection{A priori estimate and weak compactness}

In this section we collect all the results related to convergence needed for the proof of Theorem 1.
We begin with the decomposition of $f_{\epsilon}$. We write
\begin{equation*} f_{\epsilon}=\mathcal{M}(v)(\rho_{\epsilon}+\tilde{g}_{\epsilon}), \end{equation*}
where the hydrodynamic variable $\rho_{\epsilon}$ has already been defined in \ref{HV}
and $\tilde{g}_{\epsilon}$ is a deviation from the local equilibrium
state $\rho_{\epsilon}\mathcal{M}(v)$ that satisfies
\begin{equation*} \int \tilde{g}_{\epsilon} \mathcal{M}(v) \, dv =0.\end{equation*}
We also note that integrating \ref{F-P} in velocity
we obtain the hydrodynamic equation for $\rho_{\epsilon}$
\begin{equation} \label{HE} \partial_{t}\rho_{\epsilon}+\frac{1}{\epsilon} \nabla_{x} \cdot
\int \mathcal{M}(v) \nabla_{v}\tilde{g}_{\epsilon} \, dv =0. \end{equation}
We prove the following.
\begin{lemma}
Assume a mild-weak solution $f_{\epsilon}$ of \ref{F-P} with an
$L^{2}_{\mathcal{M}_{eq}}$ bound on the initial data i.e.
$\| f_{\epsilon}(0,\cdot,\cdot)\|_{L^{2}_{\mathcal{M}_{eq}}}<\infty$.
Then, there exists a sequence $\epsilon_{i}\to 0$ such that
\begin{align*} \rho_{\epsilon_{i}} & \rightharpoonup \rho \quad \text{weakly in}
\quad L^{2}(dx) \quad \forall t \geq 0,
\\ \tilde{g}_{\epsilon_{i}}& \rightharpoonup \tilde{g} \quad
\text{weakly in} \quad L^{2}(\mathcal{M}(v) dv dx) \quad \forall t \geq 0,
\\ \frac{1}{\epsilon_{i}}G^{1/2}\nabla_{v}\tilde{g}_{\epsilon_{i}}
& \rightharpoonup J \quad \text{weakly in}
\quad L^{2}(\mathcal{M}(v) dv dx dt).
\end{align*}
\end{lemma}

\begin{proof}
In order to study the limit $\epsilon \to 0$, we begin with the a priori estimate in
$L^{2}_{\mathcal{M}_{eq}}(\mathbb{R}^{n}_{x}\times \mathbb{R}^{n}_{v})$. This is an energy estimate for
$h_{\epsilon}(t,x,v)$
in $L^{2}(d\mu)$, with
$h_{\epsilon}(t,x,v):=f_{\epsilon}(t,x,v) / \mathcal{M}_{eq}$. It is achieved
by multiplying \ref{F-P} with $h_{\epsilon}$ and integrating in $d\mu$ to get
\begin{equation} \label{ApEs3}\frac{1}{2}  \int  h^{2}_{\epsilon}(t,x,v) \, d\mu +
\frac{1}{\epsilon^{2}} \int_{0}^{t}\!\!\! \int \Big| G^{1/2}(x)
\nabla_{v} h_{\epsilon}(s,x,v) \Big|^{2}  \, d\mu \, ds=
\frac{1}{2} \int h^{2}_{\epsilon}(0,x,v) \, d\mu .\end{equation}

To simplify the analysis, we consider the basic assumption
$\inf \, V(x)> - \infty$. Then, a priori
estimate \ref{ApEs3} gives the following two bounds,
\begin{equation} \label{Es1} \int \rho^{2}_{\epsilon} \, dx < \infty ,
\qquad \iint \tilde{g}^{2}_{\epsilon} \mathcal{M}(v) \, dv \, dx
< \infty \qquad \forall t\geq 0 . \end{equation}
For the first bound in \ref{Es1} we used a simple Jensen inequality on
the $L^{2}(d\mu)$ estimate for $h_{\epsilon}$.
We also have (as a result of \ref{ApEs3}) the energy bound,
\begin{equation} \label{Es2} \frac{1}{\epsilon^2}\int_{0}^{T} \!\!\!\! \iint
|G^{1/2}\nabla_{v}\tilde{g}_{\epsilon}|^{2} \mathcal{M}(v)
 \, dv \, dx \,ds < \infty \quad \text{for any} \quad T>0.\end{equation}
Based on \ref{Es1} \& \ref{Es2}, and after picking a sequence $\epsilon_{i}\to 0$, we can extract
a subsequence which without loss of generality we still call $\epsilon_{i}$ so that
the convergences in the statement of the lemma  hold.
\end{proof}

It is important to comment that we want something stronger than just $\rho_{\epsilon}$
being weakly compact in $L^{2}(dx) \quad \forall t \geq 0$. We actually want a
uniform (in time) type of convergence, so that we don't have
a problem when we later pass to the limit in integrals of
time. For this reason, we prove that $\rho_{\epsilon}$
is compact in $C([0,T],w-L^{2}(dx))$ in the lemma that follows.

\begin{lemma} Under the assumptions of Theorem 1,
$\rho_{\epsilon}$ is compact in $C([0,T],\text{w}-L^{2}(dx))$ i.e.
\begin{equation*} \rho_{\epsilon} \rightharpoonup \rho \quad \text{in}
\quad C([0,T],\text{w}-L^{2}(dx)). \end{equation*}
\end{lemma}

\begin{proof}
Consider the functional $H(t)=\int \phi(x) \rho_{\epsilon}(t,x) dx$ ,
$0<t<T$, for a fixed $T>0$ and $\phi \in C_{c}^{\infty}(\mathbb{R}_{x}^{n})$.
$H(t)$ can be proven to be pointwise finite for any $0<t<T$, using the
Cauchy-Schwartz inequality and always assuming finite initial energy.

Now, if we consider $t_{1},t_{2}>0$ such that $0 \leq t_{1} \leq t_{2} \leq T$, we have
\begin{align*}& H(t_{2})-H(t_{1}) =\int
\phi(x) \rho_{\epsilon}(t,x)\Big|_{t=t_{1}}^{t=t_{2}} \, dx \qquad (\text{Use weak form of
\eqref{HE}})\\
&=\frac{1}{\epsilon} \int_{t_{1}}^{t_{2}}\!\!\!\! \iint
\nabla_{x}\phi(x)\cdot \nabla_{v}\tilde{g}_{\epsilon}
\mathcal{M}(v) \, dv \, dx \, ds \\
&\leq \!\! \left( \int_{t_{1}}^{t_{2}} \!\!\!\! \iint |G^{-1/2}
\nabla_{x}\phi(x) |^{2} \mathcal{M}(v) \, dv  dx  ds \! \right)
^{\frac{1}{2}} \!\!\! \left( \int_{t_{1}}^{t_{2}} \!\!\!\! \iint \frac{|G^{1/2}
\nabla_{v}\tilde{g}_{\epsilon}|^{2}}{\epsilon^{2}} \mathcal{M}(v)
\, dv  dx ds \! \right)^{\frac{1}{2}} \\ & \leq (t_{2}-t_{1})^{\frac{1}{2}}
\left( \int |G^{-1/2}\nabla_{x}\phi(x)|^{2}
dx \right)^{\frac{1}{2}} \left( \int_{t_{1}}^{t_{2}}\!\!\!\!
\iint\frac{|G^{1/2}\nabla_{v}\tilde{g}_{\epsilon}|^{2}}
{\epsilon^{2}} \mathcal{M}(v) \, dv \, dx \, ds \right)^{\frac{1}{2}}
\\ & \leq C(t_{2}-t_{1})^{\frac{1}{2}}. \end{align*}

The Arzel\'a-Ascoli theorem states that pointwise boundedness
and equicontinuity suffice to show that the family $\int \phi(x) \rho_{\epsilon}(t,x)\, dx$
is  compact in $C([0,T])$ for a given function $\phi \in
C^{\infty}_{c}(\mathbb{R}^{n}_{x})$. Notice that condition
$G^{-1}(x)\in (L^{1}_{loc}(\mathbb{R}^{n}_{x}))^{n\times n}$ is important so
that the first integral is finite.

Next, we use a standard density argument to show that
$\int \phi(x) \rho_{\epsilon}(t,x)\, dx$ is compact in $C([0,T])$ for $\phi \in C_{c}(\mathbb{R}^{n}_{x})$ . Since
$C_{c}(\mathbb{R}^{n}_{x})$ is now a separable space, separability will allow us to make use of Cantor's
diagonal argument and extract a subsequence $\rho_{\epsilon_{j}}$ so that
\begin{equation*} \int \phi(x)\rho_{\epsilon_{j}}(t,x) \, dx
\to \int \phi(x)\rho(t,x) \, dx  \qquad  \text{as} \quad j \to \infty, \end{equation*}
for any $\phi$ in a countable subset of $C_{c}(\mathbb{R}^{n}_{x})$, and uniformly on $[0,T]$. This last
convergence can be extended to any $\phi \in C_{c}(\mathbb{R}^{n}_{x})$ again by use of a density argument.

We close by approximating any function $\phi \in L^{2}(\mathbb{R}^{n}_{x})$ by a sequence
$\phi_{m} \in C_{c}(\mathbb{R}^{n}_{x})$, so that $\phi_{m} \to \phi$
a.e. and $\|\phi_{m}-\phi\|_{L^{2}}\to 0$. This
way, we show
\begin{equation*} \int \rho_{\epsilon}(t,x)(\phi(x)-\phi_{m}(x))\, dx \to 0
\qquad \text{as} \quad m \to 0,\end{equation*}
uniformly in $\epsilon >0$ and $[0,T]$. This yields that $\rho \in L^{2}(\mathbb{R}^{n}_{x})$ and that
$\rho_{\epsilon}$ is compact in $C([0,T],w-L^{2}(\mathbb{R}^{n}_{x}))$. (see \cite{Go}).
\end{proof}

\subsection{Passage to the limit}

Now that weak compactness of $\rho_{\epsilon}$ has been established uniformly in $[0,T]$,
we can proceed with the derivation of an equation for the
deviation $\tilde{g}_{\epsilon}$, i.e.
\begin{align} \label{RemEq}\epsilon \partial_{t} \tilde{g}_{\epsilon} -\nabla_{x} \cdot \int
\mathcal{M}\nabla_{v}\tilde{g}_{\epsilon} \, dv &+ v \cdot \left(\nabla_{x}(\rho_{\epsilon}+\tilde{g}_{\epsilon})
+\nabla V(x)(\rho_{\epsilon}+\tilde{g}_{\epsilon})\right)
\\ \nonumber &-\nabla V(x) \cdot \nabla_{v}\tilde{g}_{\epsilon}=\frac{1}{\epsilon} \frac{1}{\mathcal{M}}
\nabla_{v}\cdot(\mathcal{M}G(x)\nabla_{v}\tilde{g}_{\epsilon})
. \end{align}

A mild solution of \ref{RemEq} will be in $C(\mathbb{R}_{+},
\mathcal{D}'(\mathbb{R}^{n}_{x}\times \mathbb{R}^{n}_{v}))$. The
weak formulation is given by the expression
\begin{align} \nonumber &\epsilon  \iint \mathcal{M}(v) \varphi
\left(\tilde{g}_{\epsilon}(t_{2})-\tilde{g}_{\epsilon}(t_{1})\right) \, dv dx
 \\  \nonumber & \qquad +\int_{t_{1}}^{t_{2}}\!\!\!\! \iint \mathcal{M}(v) \nabla_{x}\varphi \cdot \left( \int
\mathcal{M}(v') \nabla_{v'}\tilde{g}_{\epsilon}\, dv'\right) \, dv  dx  ds \\
 \nonumber & \qquad \qquad +  \int_{t_{1}}^{t_{2}} \!\!\!\! \iint
 \mathcal{M}(v) v \cdot \left( -\nabla_{x} \varphi \,
 \rho_{\epsilon}+
\varphi \nabla V(x) \, \rho_{\epsilon}\right) \, dv  dx  ds
\\  \nonumber & \qquad \qquad \qquad  + \int_{t_{1}}^{t_{2}}\!\!\!\!
\iint \mathcal{M}(v)  \left( -\nabla_{x} \varphi \cdot
\nabla_{v}\tilde{g}_{\epsilon}+
\varphi \nabla V(x) \cdot \nabla_{v}\tilde{g}_{\epsilon}\right) \, dv  dx  ds
\\\label{WeakFP} &-\!\! \int_{t_{1}}^{t_{2}} \!\!\!\! \iint \mathcal{M}(v) \varphi \nabla V(x)
\cdot \nabla_{v} \tilde{g}_{\epsilon} \, dv  dx  ds
=- \frac{1}{\epsilon} \int_{t_{1}}^{t_{2}} \!\!\!\! \iint
\mathcal{M}(v) \nabla_{v}\varphi \cdot G \nabla_{v}\tilde{g}_{\epsilon} \, dv  dx  ds ,\end{align}
where $\varphi(x,v) \in C^{\infty}_{c}(\mathbb{R}^{n}_{x}\times \mathbb{R}^{n}_{v})$.
In the lemma that follows, we show what happens when we let $\epsilon \to 0$ in \ref{WeakFP}.

\begin{lemma}
Under the assumptions of Theorem 1, in the limit $\epsilon \to 0$
the limiting functions $\rho$ and $J$ (from Lemma 1) satisfy
\begin{align} \label{L-RemEq}\int_{t_{1}}^{t_{2}}\!\!\!\! \iint
\mathcal{M}(v)v & \cdot (-\nabla_{x}\varphi \, \rho + \varphi \nabla V(x) \, \rho) \, dv \,dx \, ds \\ \nonumber &=
-\int_{t_{1}}^{t_{2}}\!\!\!\! \iint \mathcal{M}(v)\nabla_{v}\varphi \cdot G^{1/2} J \, dv\, dx\,  ds
\qquad \forall \varphi \in C^{\infty}_{c}(\mathbb{R}^{n}_{x}\times \mathbb{R}^{n}_{v}) .\end{align}
\end{lemma}

\begin{proof}
We use the notation $I_{j}$ ($1 \leq j \leq 6$) for the successive integral terms that appear in the
weak formulation \ref{WeakFP} in their order of appearance.
The study of the order of magnitude for each of them reveals that in the limit $\epsilon \to 0$
only terms $I_{3} \, \& \, I_{6}$ do not vanish. In all the estimates
that follow we use \ref{Es1} \& \ref{Es2}, so that we have
\begin{align*} I_{1} & = \epsilon \iint \mathcal{M}(v) \varphi
\left(\tilde{g}_{\epsilon}(t_{2})-\tilde{g}_{\epsilon}(t_{1})\right) \, dv \, dx
\\ &\leq \epsilon \left(\iint \varphi^{2}\mathcal{M}(v) \, dv \, dx \right)^{1/2}
\left( \iint (|\tilde{g}_{\epsilon}(t_{2})|^{2}+|\tilde{g}_{\epsilon}
(t_{1})|^{2})\mathcal{M}(v) \, dv \, dx \right)^{1/2}
\\ & \leq C \epsilon =O(\epsilon) . \end{align*}
\begin{align*} I_{2} &=\int_{t_{1}}^{t_{2}}\!\!\!\!\iiint
\mathcal{M}(v)\mathcal{M}(v') \nabla_{x}\varphi(x,v,s)\cdot
\nabla_{v'}\tilde{g}_{\epsilon}(x,v',s)
\, dv' \, dv \, dx \, ds \\ & \leq \epsilon \int_{t_{1}}^{t_{2}}\!\!\!\! \iiint \mathcal{M}(v)\mathcal{M}(v')
|G^{-1/2}\nabla_{x}\varphi|
\frac{|G^{1/2}\nabla_{v'}\tilde{g}_{\epsilon}|}{\epsilon}
\, dv' \, dv \, dx \, ds \\ & \leq \epsilon \left( \int_{t_{1}}^{t_{2}} \!\!\!\! \iint \mathcal{M}(v)
|G^{-1/2}\nabla_{x}\varphi|^{2} \, dv \, dx \, ds \right)^{1/2}
\\ & \qquad \qquad \times \left(\frac{1}{\epsilon^{2}} \int_{t_{1}}^{t_{2}} \!\!\!\! \iint \mathcal{M}(v')
|G^{1/2}\nabla_{v'}\tilde{g}_{\epsilon}|^{2}\, dv' \, dx \, ds \right)^{1/2}
\!\! \leq C \epsilon=O(\epsilon).  \end{align*}
\begin{align*} I_{3} & =  \int_{t_{1}}^{t_{2}}\!\!\!\! \iint \mathcal{M}(v)
v \cdot \left( -\nabla_{x} \varphi \, \rho_{\epsilon}+
\varphi \nabla V(x) \, \rho_{\epsilon}\right) \, dv \, dx \, ds
\\& \leq \left( \int_{t_{1}}^{t_{2}} \!\!\!\! \iint \mathcal{M}(v) |v|^{2}
(|\nabla_{x}\varphi|^{2}+|\varphi \nabla V(x)|^{2})
\, dv \, dx \, ds \! \right)^{1/2} \!\!
\left( \int_{t_{1}}^{t_{2}} \!\!\!\! \int  \rho^{2}_{\epsilon} \, dx \, ds \!\right)^{1/2} \\ & \leq C
\left( \iint \mathcal{M}(v) |v|^{2}(|\nabla_{x}\varphi|^{2}+|\varphi \nabla V(x)|^{2})
\, dx \right)^{1/2}
\!\!\left(  \int_{t_{1}}^{t_{2}}\!\!\!\! \int  \rho^{2}_{\epsilon}
\, dx \, ds \right)^{1/2} \!\!\! =O(1).\end{align*}
\begin{align*}  I_{4} & = \int_{t_{1}}^{t_{2}} \!\!\!\! \iint \mathcal{M}(v)  \left( -\nabla_{x} \varphi \cdot
\nabla_{v}\tilde{g}_{\epsilon}+
\varphi \nabla V(x) \cdot \nabla_{v}\tilde{g}_{\epsilon}\right) \, dv \, dx \, ds
\\& \leq \epsilon \left( \int_{t_{1}}^{t_{2}}\!\!\!\! \iint \mathcal{M}(v)
(|G^{-1/2}\nabla_{x}\varphi|^{2}+|\varphi G^{-1/2}\nabla V(x)|^{2}) \, dv \, dx \, ds \right)^{1/2}
 \\ & \qquad \qquad \qquad  \times \left( \int_{t_{1}}^{t_{2}} \!\!\!\! \iint
 \frac{1}{\epsilon^{2}}|G^{1/2}\nabla_{v}\tilde{g}_{\epsilon}|^{2}
\mathcal{M}(v) \, dv \, dx \, ds \right)^{1/2} \!\! \leq C \epsilon =O(\epsilon). \end{align*}
\begin{align*} I_{5} & = -\int_{t_{1}}^{t_{2}} \!\!\!\! \iint \mathcal{M}(v) \varphi \nabla V(x) \cdot \nabla_{v}
\tilde{g}_{\epsilon} \, dv \, dx \, ds
\\ & \leq \epsilon \left( \int_{t_{1}}^{t_{2}} \!\!\!\! \iint \mathcal{M}(v)
|\varphi G^{-1/2}\nabla V(x)|^{2} \, dv \, dx \, ds \right)^{1/2}
 \\ & \qquad \qquad \qquad \times \left( \int_{t_{1}}^{t_{2}} \!\!\!\! \iint
 \frac{1}{\epsilon^{2}}|G^{1/2}\nabla_{v}\tilde{g}_{\epsilon}|^{2}
\mathcal{M}(v) \, dv \, dx \, ds \right)^{1/2} \!\! \leq C \epsilon =O(\epsilon). \end{align*}
\begin{align*} I_{6} & =  \frac{1}{\epsilon} \int_{t_{1}}^{t_{2}} \!\!\!\! \iint
\mathcal{M}(v) \nabla_{v}\varphi \cdot G \nabla_{v}\tilde{g}_{\epsilon} \, dv \, dx \, ds
\\ &\leq \left( \int_{t_{1}}^{t_{2}} \!\!\!\! \iint |G^{1/2}\nabla_{v} \varphi|^{2}\mathcal{M}(v)
  \, dv \, dx \, ds \right)^{1/2}
\\ & \qquad \qquad \qquad \times \left( \int_{t_{1}}^{t_{2}} \!\!\!\! \iint
\frac{1}{\epsilon^{2}}|G^{1/2}\nabla_{v}\tilde{g}_{\epsilon}|^{2}
\mathcal{M}(v) \, dv \, dx \, ds \right)^{1/2}\!\! =O(1).
\end{align*}
Now that we have established all the above bounds, we take $\epsilon \to 0$
and use the convergence results in Lemmas $1$ \& $2$ to derive \ref{L-RemEq}.
\end{proof}

\textbf{Proof of Theorem 1.}
We write the hydrodynamic equation \ref{HE} for $\rho_{\epsilon}$ in its
weak form, and take the limit $\epsilon \to 0$ to obtain
\begin{equation} \label{W-HE}\int \phi(\cdot) \rho(t,\cdot) \Big|_{t=t_{1}}^{t=t_{2}} \, dx =
\int_{t_{1}}^{t_{2}} \!\!\!\! \iint \mathcal{M}(v) \nabla_{x}\phi
\cdot G^{-1/2} J \, dv \, dx \, ds \qquad
\forall \phi \in C^{\infty}_{c}(\mathbb{R}^{n}_{x}). \end{equation}
In order to give the limiting equation for $\rho(t,x)$
we should combine \ref{L-RemEq} \& \ref{W-HE}. The two equations can be coupled for the choice of test
function $\varphi(x,v)=\nabla_{x}\phi \cdot G^{-1}v$, where $\phi \in C^{\infty}_{c}(\mathbb{R}^{n}_{x})$.
The only problem is that this function is not smooth or
compactly supported in $\mathbb{R}^{n}_{v}$, so we have to modify it slightly
(for non smooth $G^{-1}(x)$ regularization in $x$ is also needed).

We begin by taking the cut-off function $\chi_{\delta_{1}}(v)=\chi(\delta_{1}v)$, where $\chi(v) \in
C^{\infty}_{c}(\mathbb{R}^{n}_{v})$ is
a function with values $0 \leq \chi(v) \leq 1$ such that $\chi(v)=1 \,
\text{for} \, |v|\leq 1$ and $\chi(v)=0 \, \text{for} \, |v|\geq 2$.
We also consider the standard mollification function,
\begin{equation*}\eta_{\delta_{2}}(v)=\frac{1}{\delta^{n}_{2}}\eta
\left( \frac{v}{\delta_{2}} \right) \quad \text{for} \quad
\eta \in C^{\infty}_{c}(\mathbb{R}^{n}_{v}) \, \, \text{such that} \,\,  \int \eta(v) \, dv=1 . \end{equation*}
We now take the function $\varphi_{\delta_{1},\delta_{2}}
(x,v)=\left(\chi_{\delta_{1}}(v)\nabla_{x}\phi \cdot G^{-1}v \right)\star \eta_{\delta_{2}}$.
A standard result for the mollified function is that $\varphi_{\delta_{1},\delta_{2}}$ converges to $\varphi$ a.e.
in $\mathbb{R}^{n}_{v}$ (as $\delta_{1},\delta_{2} \to 0$). Obviously
$\nabla_{x}\varphi_{\delta_{1},\delta_{2}}$ converges to $\nabla_{x}\varphi$ a.e.
in $\mathbb{R}^{n}_{v}$ , since the cut-off and mollification acts only in the $v$ variable.

By the substitution of $\varphi$ with $\varphi_{\delta_{1},\delta_{2}}$  in \ref{L-RemEq}, we have
\begin{align*} \int_{t_{1}}^{t_{2}}\!\!\!\! \iint \mathcal{M}(v)v \cdot
(-\nabla_{x}\varphi_{\delta_{1},\delta_{2}} \, \rho &+
 \varphi_{\delta_{1},\delta_{2}} \nabla V(x) \, \rho) \, dv \, dx \, ds=
\\ & -\int_{t_{1}}^{t_{2}}\!\!\!\! \iint \mathcal{M}(v)\nabla_{v}\varphi_{\delta_{1},\delta_{2}}
\cdot G^{1/2} J \, dv \, dx \, ds. \end{align*}
We also have \begin{align*} \nabla_{v}\varphi_{\delta_{1},\delta_{2}}(x,v)&=\nabla_{v}
\left( (\nabla_{x}\phi \cdot G^{-1}v \, \chi_{\delta_{1}}(v))\star \eta_{\delta_{2}}\right)
\\ &=\nabla_{v}(\nabla_{x}\phi \cdot G^{-1}v \, \chi_{\delta_{1}}(v))\star \eta_{\delta_{2}} \\ &=
\left( \nabla_{x}\phi \cdot G^{-1} \chi_{\delta_{1}}(v) + \nabla_{x}\phi \cdot G^{-1}v \, \nabla_{v}
\chi_{\delta_{1}}(v) \right) \star \eta_{\delta_{2}} , \end{align*}
where we use the fact that $\nabla_{v}(f \star \eta_{\delta})=\nabla_{v}f \star \eta_{\delta}$.

A typical estimate for $\nabla_{v}\chi_{\delta_{1}}(v)$ is $|\nabla_{v}\chi_{\delta_{1}}(v)|\leq C \delta_{1}$. This
can be easily seen by the definition of $\chi_{\delta_{1}}$ and the fact that $|\nabla_{v}\chi|\leq C$ for some $C>0$,
since
$\chi \in C^{\infty}_{c}(\mathbb{R}^{n}_{v})$. This estimate, together with the computation of
$\nabla_{v}\varphi_{\delta_{1},\delta_{2}}(x,v)$ above and the
dominated convergence theorem imply that in the limit $\delta_{1}, \delta_{2} \to 0$, we
actually have that \ref{L-RemEq} holds with
$\varphi(x,v)=\nabla_{x}\phi(x) \cdot G^{-1}v$. This choice of test function allows the coupling
of \ref{W-HE} and \ref{L-RemEq} that yields
\begin{equation*}\int \phi(\cdot) \rho(t,\cdot) \Big|_{t=t_{1}}^{t=t_{2}} \, dx
=\int_{t_{1}}^{t_{2}}\!\!\!\! \int
\left( \nabla_{x} \cdot (G^{-1}\nabla_{x}\phi)+\nabla_{x}\phi
\cdot G^{-1}\nabla V(x) \right) \rho \, dx \, ds , \end{equation*}
which is the weak form of \ref{Sm}. This completes the proof of Theorem 1. $\Box$

\section{Diffusive limit via relative entropy and proof of Theorem 2}

To prove Theorem $2$, we begin with the a priori estimates that will be used later to
show that the remainder term $r_{\epsilon}$ is of order $O(\epsilon)$. The exact formula for
$r_{\epsilon}$ is given in the study of the
evolution of $H(f_{\epsilon}|\rho \mathcal{M})$ in Section 4.2.

\subsection{A priori estimates}

The following proposition contains all the estimates needed for a solution $f_{\epsilon}$.

\begin{proposition}
Assume $f_{\epsilon}$ is a solution of \ref{F-P}, with initial data satisfying \ref{ApEs2}. Let also
$0<T<\infty$. Then, the following hold:
\\ (i) $f_{\epsilon}(1+V(x)+|v|^{2}+\ln f_{\epsilon})$ is bounded in $L^{\infty}((0,T),L^{1}
(\mathbb{R}^{n}_{x}\times \mathbb{R}^{n}_{v}))$,
\\ (ii) $\frac{1}{\epsilon}G^{1/2}(x)(v \sqrt{f_{\epsilon}}+2\nabla_{v}\sqrt{f_{\epsilon}})$ is in
$L^{2}((0,T),\mathbb{R}^{n}_{x}\times \mathbb{R}^{n}_{v})$,
\\ (iii) $|v|^{2}f_{\epsilon}$ is bounded in  $L^{\infty}((0,T),L^{1}
(\mathbb{R}^{n}_{x}\times \mathbb{R}^{n}_{v}))$.
\end{proposition}

\begin{proof}
(i) We introduce the free energy associated with the F-P equation \ref{F-P},
\begin{equation*} \mathcal{E}(f_{\epsilon}):= \iint f_{\epsilon}
\left( \ln{f_{\epsilon}} +\frac{|v|^{2}}{2}+V(x) \right)\, dv \, dx .\end{equation*}
The free energy is dissipated since
\begin{equation*} \frac{d}{dt}\mathcal{E}(f_{\epsilon})=
-\frac{1}{\epsilon^{2}}\iint |d_{\epsilon}|^{2} \, dv \, dx , \end{equation*}
where \begin{equation*} d_{\epsilon}= G^{1/2}(x)(v \sqrt{f_{\epsilon}}
+2\nabla_{v}\sqrt{f_{\epsilon}})=2 \sqrt{\mathcal{M}} G^{1/2}(x)\nabla_{v}
\sqrt{\frac{f_{\epsilon}}{\mathcal{M}}}. \end{equation*}
The dissipation of energy implies
\begin{equation} \label{Es3} \mathcal{E}(f_{\epsilon}(T,\cdot,\cdot)) +\frac{1}{\epsilon^{2}}
\int_{0}^{T}\!\!\!\! \iint |d_{\epsilon}|^{2} \, dv \, dx \, ds
= \mathcal{E}(f_{\epsilon}(0,\cdot,\cdot)) .\end{equation}
(ii) It also follows from \eqref{Es3} that $d_{\epsilon}$ is of order
$O(\epsilon)$ in $L^{2}$ i.e.
\begin{equation*} \int_{0}^{T}\!\!\!\! \iint |d_{\epsilon}|^{2}
\, dv \, dx \, dt \leq C \epsilon^{2} , \qquad T>0 .\end{equation*}
(iii) Let us now establish a bound for $\iint |v|^{2} f_{\epsilon}\, dv \, dx$
in $L^{\infty}(0,T)$.
This bound (uniform in time) is a straightforward consequence
of the elementary Frenchel-Young inequality
\begin{equation*} ab \leq h(a)+h^{*}(b),\end{equation*}
where $h$,$h^{*}$ are a Young's convex pair ($h^{*}$ is explicitly computed
by the Legendre transform of the convex function $h$). Here, we use $h(z)=z\log z$
and $h^{*}(z)=e^{z-1}$, i.e.
\begin{equation*} \frac{1}{4}\iint |v|^{2}f_{\epsilon}\, dv \, dx \leq
\iint f_{\epsilon}\log \frac{f_{\epsilon}}{\mathcal{M}_{eq}} \, dv \, dx + \iint
e^{\frac{|v|^{2}}{4}-1} \mathcal{M}_{eq}\, dv \, dx .\end{equation*}
This implies that
\begin{equation*} \iint |v|^{2}f_{\epsilon}(t,v,x) \, dv \, dx \leq C
\qquad \text{for some} \, \, C>0, \, \, \, \forall t \in [0,T],\end{equation*}
since $e^{-V(x)}\in L^{1}(\mathbb{R}^{n}_{x})$ and the entropy integral is
bounded by the a priori estimate.
\end{proof}

\subsection{Evolution of the relative entropy}

We now prove the following result for the evolution of the
relative entropy.
\begin{lemma}
Assume a sufficiently regular solution $f_{\epsilon}$ of equation \ref{F-P}, with
initial data $f_{\epsilon}(0,\cdot, \cdot)$
satisfying \eqref{ApEs2}. It is shown
that \begin{equation}\label{In1} H(f_{\epsilon}(t,\cdot,\cdot)|\rho(t,\cdot) \mathcal{M})
\leq  H(f_{\epsilon}(0,\cdot,\cdot)|\rho(0,\cdot) \mathcal{M})
+\int_{0}^{t}r_{\epsilon}(s) \, ds , \end{equation}
for a remainder term $r_{\epsilon}$ given explicitly by
\begin{equation} \label{RemEq2} r_{\epsilon}(t)=- \int \left( \epsilon \partial_{t}J_{\epsilon}
+ \nabla_{x}\cdot \left( \int \mathcal{M} \nabla_{v}
\left( \frac{f_{\epsilon}}{\mathcal{M}}\right)
\otimes v \, dv \right)\right)
\cdot G^{-1} \left( \frac{\nabla \rho}{\rho}
+\nabla V(x)\right)\, dx . \end{equation}
\end{lemma}

\begin{proof}
We start with the computation of the evolution of the
$H(\rho_{\epsilon}|\rho)$ relative entropy. This computation
becomes partly obsolete later when we perform a similar computation for $H(f_{\epsilon}|\rho \mathcal{M})$.
Nevertheless, we begin with computing $\frac{d}{dt}H(\rho_{\epsilon}|\rho)$,
especially since it contains parts important in the computation of
$\frac{d}{dt}H(f_{\epsilon}|\rho \mathcal{M})$. Hence,
\begin{align*}
&\frac{d}{dt}H(\rho_{\epsilon}|\rho)=\frac{d}{dt}\int
\rho_{\epsilon}\log{\frac{\rho_{\epsilon}}{\rho}} \, dx=
\frac{d}{dt}\int \rho_{\epsilon} \log{\rho_{\epsilon}} \, dx
-\frac{d}{dt} \int \rho_{\epsilon} \log{\rho} \, dx\\
&= \int \partial_{t}\rho_{\epsilon} (\log{\rho_{\epsilon}}+1)
\, dx -\int \partial_{t}\rho_{\epsilon} \log{\rho} \, dx
-\int \frac{\rho_{\epsilon}}{\rho} \partial_{t}\rho \, dx \qquad (\text{Use \eqref{Eq-Rhoe},\eqref{Sm}})\\
&=\frac{1}{\epsilon}\int J_{\epsilon} \cdot \frac{\nabla \rho_{\epsilon}}
{\rho_{\epsilon}} \, dx -\frac{1}{\epsilon}\int J_{\epsilon} \cdot
\frac{\nabla \rho}{\rho} \, dx -\int \frac{\rho_{\epsilon}}{\rho}
\nabla \cdot \left( G^{-1}(\nabla \rho +\nabla V(x) \rho) \right) \, dx \\
&= \frac{1}{\epsilon}\int J_{\epsilon} \cdot \left( \frac{\nabla \rho_{\epsilon}}
{\rho_{\epsilon}}-\frac{\nabla \rho}{\rho}\right)\, dx+ \int \rho \nabla
\left( \frac{\rho_{\epsilon}}{\rho}\right) \cdot G^{-1}
\left( \frac{\nabla \rho}{\rho} + \nabla V(x) \right) \, dx\\
&= \frac{1}{\epsilon}\int J_{\epsilon} \cdot \left( \frac{\nabla \rho_{\epsilon}}
{\rho_{\epsilon}}-\frac{\nabla \rho}{\rho}\right)\, dx + \int
\left( \frac{\nabla \rho_{\epsilon}}{\rho_{\epsilon}}-\frac{\nabla \rho}{\rho} \right)\cdot G^{-1}
\left( \frac{\nabla \rho}{\rho} + \nabla V(x) \right) \rho_{\epsilon}\, dx \\
&= \int \left( \frac{\nabla \rho_{\epsilon}}{\rho_{\epsilon}}-\frac{\nabla \rho}{\rho} \right) \cdot
\left( \frac{1}{\epsilon}\frac{J_{\epsilon}}{\rho_{\epsilon}}+ G^{-1}
\left( \frac{\nabla \rho}{\rho} + \nabla V(x) \right)\!\! \right) \rho_{\epsilon}\, dx \\
& = \! - \! \int \!\! G \left( \frac{1}{\epsilon}\frac{J_{\epsilon}}{\rho_{\epsilon}}+ G^{-1}
\left( \frac{\nabla \rho}{\rho} + \nabla V(x) \right)\!\! \right)
\! \cdot \! \left(\frac{1}{\epsilon} \frac{J_{\epsilon}}{\rho_{\epsilon}}+ G^{-1}
\left( \frac{\nabla \rho}{\rho} + \nabla V(x) \right)\!\! \right) \rho_{\epsilon} \, dx \\
& +r'_{\epsilon}=- \int \Big| \frac{1}{\epsilon} G^{1/2} \frac{J_{\epsilon}}{\rho_{\epsilon}} +G^{-1/2}\left(
\frac{\nabla \rho}{\rho} +\nabla V(x)\right)\Big|^{2} \rho_{\epsilon} \, dx +r'_{\epsilon}.
\end{align*}
In the second to last equality, we made use of
\begin{align} \nonumber \frac{\nabla \rho_{\epsilon}}{\rho_{\epsilon}}-\frac{\nabla \rho}{\rho}
=-\frac{1}{\epsilon} G \frac{J_{\epsilon}}{\rho_{\epsilon}}-
\frac{\nabla \rho}{\rho}- & \nabla V(x)
-\epsilon \frac{\partial_{t}J_{\epsilon}}{\rho_{\epsilon}}
\\ \label{E8} -& \frac{1}{\rho_{\epsilon}} \nabla_{x} \cdot \int \mathcal{M} \nabla_{v}
\left( \frac{f_{\epsilon}}{\mathcal{M}} \right) \otimes v \, dv ,
\end{align}
which is derived directly from \ref{Eq-Je2}.
The remainder term $r'_{\epsilon}$ equals
\begin{equation*} r'_{\epsilon}= \! - \! \int \!\! \left( \epsilon \partial_{t}J_{\epsilon} +\nabla_{x}\! \cdot \!\!
\int \mathcal{M} \nabla_{v}\left( \frac{f_{\epsilon}}{\mathcal{M}}\right) \otimes v \, dv
\right) \! \cdot \! \left( \frac{1}{\epsilon}\frac{J_{\epsilon}}{\rho_{\epsilon}}+ G^{-1}
\left( \frac{\nabla \rho}{\rho} + \nabla V(x) \right)\!\! \right)  dx .\end{equation*}
\begin{remark}
Notice that $r'_{\epsilon}$ is a remainder term that should vanish as $\epsilon \to 0$.
We do not bother with showing that $r'_{\epsilon} \to 0$ in rigorous manner, as we mainly work with
the relative entropy $H(f_{\epsilon}|\rho \mathcal{M})$.
Yet, as we remark in the end of Section 4, the computation of
$\frac{d}{dt}H(\rho_{\epsilon}|\rho)$ alone can be used to establish the convergence of $f_{\epsilon}$
that we prove in Theorem 2.
\end{remark}

At this point, we compute the evolution of $H(f_{\epsilon}| \rho \mathcal{M})$
in similar manner. To make things easier we can introduce
the global equilibrium state $\mathcal{M}_{eq}(x,v)$ in the computation that follows
\begin{align} \nonumber H(f_{\epsilon}|\rho \mathcal{M})&=\iint f_{\epsilon}
\log{f_{\epsilon}} \, dv \, dx -\iint f_{\epsilon} \log{(\rho \mathcal{M})} \, dv \, dx
\\ \nonumber &= \iint f_{\epsilon} \log{\frac{f_{\epsilon}}{\mathcal{M}_{eq}}} \, dv \, dx + \iint f_{\epsilon}
\log{\frac{\mathcal{M}_{eq}}{\rho \mathcal{M}}} \, dv \, dx
\\ \nonumber &= \iint f_{\epsilon} \log{\frac{f_{\epsilon}}{\mathcal{M}_{eq}}} \, dv \, dx +\int \rho_{\epsilon}
\log{\frac{e^{-V(x)}}{\rho}} \, dx
\\ \label{RelEn1}&=H(f_{\epsilon}|\mathcal{M}_{eq})-\int \rho_{\epsilon} \log{\rho} \, dx -\int \rho_{\epsilon}V(x) \, dx
.\end{align}
The reason we introduced $H(f_{\epsilon}|\mathcal{M}_{eq})$ is
that the term $\frac{d}{dt}H(f_{\epsilon}|\mathcal{M}_{eq})$
can be easily bounded by an integral involving only hydrodynamical variables.
Indeed, the time derivative of $H(f_{\epsilon}|\mathcal{M}_{eq})$ is
\begin{align} \label{RelEn2} \frac{d}{dt} \iint f_{\epsilon}&\log{\frac{f_{\epsilon}}{\mathcal{M}_{eq}}}  \, dv \, dx
=-\frac{1}{\epsilon^{2}} \iint
f_{\epsilon} \Big|G^{1/2}\nabla_{v}\log{\frac{f_{\epsilon}}
{\mathcal{M}}} \Big|^{2} \, dv \, dx \\ \nonumber &= -\frac{1}{\epsilon^{2}}
\iint f_{\epsilon} \Big|G^{1/2}\left( \frac{\nabla_{v}f_{\epsilon}}{f_{\epsilon}} +v \right) \Big|^{2} \, dv \, dx
\leq -\frac{1}{\epsilon^{2}}\int \frac{|G^{1/2}J_{\epsilon}|^{2}}{\rho_{\epsilon}} \, dx .\end{align}
The last inequality in \ref{RelEn2} is in fact due to H\"older,
\begin{align*} \int \frac{|G^{1/2}J_{\epsilon}|^{2}}{\rho_{\epsilon}} \, dx &= \int
\frac{\left( \int G^{1/2} \left(v+\frac{\nabla_{v}f_{\epsilon}}{f_{\epsilon}} \right)f_{\epsilon} \, dv
\right)^{2}}{\rho_{\epsilon}} \, dx \\ &
\leq  \iint \Big|G^{1/2}\left( \frac{\nabla_{v}
f_{\epsilon}}{f_{\epsilon}}+v \right) \Big|^{2} f_{\epsilon} \, dv \, dx . \end{align*}
Combining \ref{RelEn1} \& \ref{RelEn2}, we obtain
\begin{align*} \frac{d}{dt}H(f_{\epsilon}|\rho \mathcal{M})&=\frac{d}{dt}H(f_{\epsilon}|\mathcal{M}_{eq})-\frac{d}{dt}
\int \rho_{\epsilon} \log{\rho} \, dx
-\frac{d}{dt} \int \rho_{\epsilon} V(x) \, dx \\
& \leq - \frac{1}{\epsilon^{2}}\int \frac{|G^{1/2}J_{\epsilon}|^{2}}{\rho_{\epsilon}} \, dx
- \frac{d}{dt} \int \rho_{\epsilon} \log{\rho} \, dx - \frac{1}{\epsilon}
\int J_{\epsilon} \cdot \nabla V(x) \, dx .\end{align*}

The computation of $\frac{d}{dt} \int \rho_{\epsilon}\log{\rho} \, dx $
has been performed as a part of the computation
of $\frac{d}{dt}H(\rho_{\epsilon}|\rho)$
above. We thus have
\begin{align}
\nonumber \frac{d}{dt}&H(f_{\epsilon}|\rho \mathcal{M}) \leq -\frac{1}{\epsilon^{2}}\int
\frac{|G^{1/2}J_{\epsilon}|^{2}}{\rho_{\epsilon}} \, dx
-\frac{1}{\epsilon}\int J_{\epsilon} \cdot \nabla V(x) \, dx
-\frac{1}{\epsilon}\int J_{\epsilon} \cdot \frac{\nabla \rho}{\rho} \, dx \\ \nonumber
&+ \int \left( \frac{\nabla \rho_{\epsilon}}{\rho_{\epsilon}}
-\frac{\nabla \rho}{\rho}\right)\cdot G^{-1}
\left( \frac{\nabla \rho}{\rho} +\nabla V(x) \right) \rho_{\epsilon} \, dx \qquad (\text{Use \eqref{E8}}) \\
\nonumber &=-\int \frac{1}{\epsilon} J_{\epsilon} \cdot \left(\frac{1}{\epsilon} G \frac{J_{\epsilon}}
{\rho_{\epsilon}} +\frac{\nabla \rho}{\rho} +\nabla V(x) \right) \, dx \\ \nonumber
& -\int \left( \frac{1}{\epsilon} G \frac{J_{\epsilon}}{\rho_{\epsilon}} +
\frac{\nabla \rho}{\rho} +\nabla V(x) \right) \cdot G^{-1}
\left( \frac{\nabla \rho}{\rho}+\nabla V(x) \right)
\rho_{\epsilon} \, dx + r_{\epsilon} \\ \nonumber
&=-\int \! \left(  \frac{1}{\epsilon} G \frac{J_{\epsilon}}{\rho_{\epsilon}}
+\frac{\nabla \rho}{\rho} +\nabla V(x)\right) \! \cdot \!
\left( \frac{1}{\epsilon}\frac{J_{\epsilon}}{\rho_{\epsilon}}+G^{-1}
\left( \frac{\nabla \rho}{\rho}+\nabla V(x)
\right)\! \right) \rho_{\epsilon} \, dx + r_{\epsilon} \\ \label{E7}
&=-\int \Big| \frac{1}{\epsilon} G^{1/2}\frac{J_{\epsilon}}{\rho_{\epsilon}}
+G^{-1/2} \left( \frac{\nabla \rho}{\rho}+ \nabla V(x)
\right)\Big|^{2} \rho_{\epsilon} \, dx + r_{\epsilon},
\end{align}
with a remainder term $r_{\epsilon}$ given by \ref{RemEq2}.
Finally, we integrate \ref{E7} in time and \ref{In1} follows.
\end{proof}

\subsection{Control of the remainder term}

We already gave the formal computation
for $\frac{d}{dt}H(f_{\epsilon}|\rho \mathcal{M})$.
Our goal is to prove that $\int_{0}^{t} r_{\epsilon}(s)\, ds \to 0$.
The remainder term $r_{\epsilon}$ that we computed in Lemma $4$ consists of two parts $r_{1,\epsilon}$ and
$r_{2,\epsilon}$, which integrated in time are
\begin{equation*} \int_{0}^{T}r_{1,\epsilon} \, dt =-\epsilon \int_{0}^{T}\!\!\!\! \iint
\partial_{t}f_{\epsilon} \, v \cdot G^{-1}\left( \frac{\nabla \rho}{\rho}
+\nabla V(x)\right) \, dv \, dx \, dt ,\end{equation*}
\begin{equation*} \int_{0}^{T} r_{2,\epsilon} \, dt =\int_{0}^{T}\!\!\!\!
\iint \left(\mathcal{M}\, v \otimes \nabla_{v}
\left( \frac{f_{\epsilon}}{\mathcal{M}}\right)\right)
: \nabla \left( G^{-1}\left( \frac{\nabla \rho}{\rho}
+\nabla V(x)\right)\right) \, dv \, dx \, dt .\end{equation*}
Our task is to show that both integrals vanish as $\epsilon \to 0$.

In the process of controlling the two terms,
we introduce a new notation for expressions involving the
hydrodynamic variable $\rho$. Thus, we denote with
$D$ the tensor $D:=\nabla(G^{-1}
(\nabla \log{\rho}+\nabla V(x)))$, and with $E,F$ the vectors
$E:=G^{-1}(\nabla \log{\rho}+\nabla V(x))$ and
$F:=G^{-1}\nabla \partial_{t} \log{\rho}$. The easiest term to
control is
\begin{align}
\nonumber \Big|\int_{0}^{T}& r_{2,\epsilon}\, dt \Big| = \Big| \int_{0}^{T}\!\!\!\! \iint \mathcal{M}
v \otimes \nabla_{v}\left( \frac{f_{\epsilon}}{\mathcal{M}}\right) : D
\, dv \, dx \, dt \Big|
\\ \nonumber &\leq  \int_{0}^{T} \!\!\!\! \iint | \sqrt{f_{\epsilon}} \left( v \otimes G^{-1/2}(x)d_{\epsilon}\right)
: D | \, dv \, dx \, dt \\ \label{In2} & \leq C \epsilon \| D\|_{\infty} \! \left( \int_{0}^{T}\!\!\!\! \iint
\frac{|G^{-1/2}(x)d_{\epsilon}|^{2}}{\epsilon^{2}} \, dv \, dx \, dt\right)^{1/2}\!\!\!
\left( \int_{0}^{T}\!\!\!\! \iint  |v|^{2}f_{\epsilon}
\, dv \, dx \, dt \right)^{1/2}\!\!\!\!\!.
\end{align}
Finally, for the first term we have
\begin{align}
\nonumber \Big| \int_{0}^{T} & r_{1,\epsilon} \, dt  \Big| =  \Big| -\epsilon \iint
\left(f_{\epsilon}(T,v,x)-f_{\epsilon}(0,v,x)\right)v \cdot E \, dv \, dx
\\ \nonumber &+ \epsilon \int_{0}^{T}\!\!\!\! \iint f_{\epsilon} \, v \cdot F \, dv \, dx \, dt \Big|
\\ \nonumber & \leq \epsilon \| E \|_{\infty} \! \left(
\left( \iint \! f_{\epsilon}(T,x,v)|v|^{2}\, dv \, dx \right)^{1/2}\!\!\! +
\left( \iint \! f_{\epsilon}(0,x,v)|v|^{2}\, dv \, dx\right)^{1/2} \right)
\\ \label{In3} &+ \epsilon T^{1/2} \| F \|_{\infty} \! \left( \int_{0}^{T}\!\!\!\! \iint
f_{\epsilon}(s,x,v)|v|^{2}\, dv \, dx \, ds \right)^{1/2}\!\!.
\end{align}
We now show why the terms $D$, $E$, $F$ are in
$L^{\infty}([0,T],\mathbb{R}^{n}_{x})$.

\textbf{Stability estimates for the terms} $D$, $E$, $F$.

The control of terms $D,E,F$ is achieved by controlling
$h_{0}=\frac{\rho}{e^{-V(x)}}$ and its derivatives. Notice that the $L^{\infty}$ bound
on $\log \frac{\rho}{e^{-V(x)}}$ implies a bound of the
type $a e^{-V(x)} \leq \rho \leq A e^{-V(x)}$ for $a,A>0$. Such
control of $h_{0}$ is a direct consequence of the maximum principle for the Smoluchowski equation \ref{Sm}. Indeed,
if $a<h_{0}(0,x)<A$ it follows by the maximum principle that $a<h_{0}(t,x)<A$ for all $t \in [0,T]$, under the
condition that $\nabla \cdot (G^{-1}(x)\nabla V(x))<\infty$ and given that
$\sup_{0 \leq t \leq T} \limsup_{x \to \infty}| h_{0}(t,x)|\leq C_{0}$.
The control of derivatives also follows from a parabolic maximum principle as we prove in

\begin{lemma}
Let $\rho(t,x)$ be a solution to the Smoluchowski equation
\ref{Sm} with initial data $\rho(0,x)$. Assume also that conditions \ref{A1}-\ref{A3} are satisfied.
It follows that $D$, $E$, $F$ are in
$L^{\infty}([0,T],\mathbb{R}^{n}_{x})$. More precisely,
$\|D(t,.)\|_{L^{\infty}(\mathbb{R}^{n}_{x})} \leq C \|D(0,.)\|_{L^{\infty}
(\mathbb{R}^{n}_{x})}$ for $0 \leq t \leq T$, with similar estimates for $E$ and $F$.
\end{lemma}

\begin{proof}
First, we define $h_{k}:=\nabla^{k}\frac{\rho}{e^{-V(x)}}$ and we want
to prove that $\| h_{k} \|_{L^{\infty}(\mathbb{R}^{n}_{x})}$ remains
bounded on the interval $[0,T]$ for $0\leq k \leq 3$. It is enough to show that
$\| h_{k}(t,.)\|_{L^{\infty}}\leq C \| h_{k}(0,.)\|_{L^{\infty}}$
for $0 \leq t \leq T$, with the constant $C$ depending on $T$ and
other constants from the bounds in \ref{A1}.

The time evolution of $h_{0}$ is given by
\begin{align*}   \partial_{t} \left(\frac{\rho}{e^{-V(x)}}\right)
& =\frac{\nabla \cdot \left(e^{-V(x)}G^{-1}(x)
\nabla \frac{\rho}{e^{-V(x)}}\right)}{e^{-V(x)}} =
\\ &\nabla \cdot \left(G^{-1}(x)\nabla \frac{\rho}{e^{-V(x)}}\right) -
\nabla V(x) \cdot G^{-1}(x)\nabla \frac{\rho}{e^{-V(x)}}. \end{align*}
In the notation we introduced for $h_{0}$, this is written as
\begin{equation} \label{h_0} \partial_{t}h_{0}=\nabla \cdot (G^{-1}(x) \nabla h_{0})
-\nabla V(x)\cdot G^{-1}(x)\nabla h_{0} .\end{equation}
Differentiating equation \ref{h_0} $m$ times, taking the
inner product (for tensors) with $h_{m}$ and integrating by parts
we obtain an $L^{2}$ estimate. As a matter of fact, we can get an $L^{p}$ theory
(for any $p>1$) and as a result
a maximum principle for $|h_{m}|$ given that we have
the appropriate control of the coefficients.

For instance, for $h_{0}$ the $L^{p}$ estimate is
\begin{align*} \frac{d}{dt}\int h^{p}_{0}\, dx &+p(p-1)\int h^{p-2}_{0}\, \nabla h_{0}
\cdot G^{-1}(x)\nabla h_{0}\, dx \\ &+\int h^{p}_{0}\, \nabla
\cdot (G^{-1}(x)\nabla V(x))\, dx=0, \qquad  p>1, \end{align*}
which under the assumptions $\|\nabla \cdot (G^{-1}\nabla V(x)) \|_{L^{\infty}}<\infty$ and
$G^{-1}(x)\geq \lambda \Id$ yields $\|h_{0}(t,.) \|_{L^{\infty}}
\leq C \|h_{0}(0,.) \|_{L^{\infty}}$ for $0 \leq t \leq T$. It should be noted that
for a divergence free or identity hydrodynamic mobility, the first
condition translates to $|\nabla^{2}V(x)|<C$ for the potential
$V(x)$.

With a bit more work we obtain (see \cite{LeBrLi2})
\begin{equation} \label{h_1} \partial_{t}\left( \frac{|h_{1}|^{2}}{2}\right)-\nabla \cdot
\left( G^{-1}(x)\nabla \left( \frac{|h_{1}|^{2}}{2}\right)\right)
+\nabla V(x)\cdot G^{-1}(x)\nabla \left(\frac{|h_{1}|^{2}}{2} \right)\leq C|h_{1}|^{2},\end{equation}
where the constant $C$ now depends on $\|\nabla G^{-1}\|_{L^{\infty}}$ and
$\|\nabla (G^{-1}\nabla V(x))\|_{L^{\infty}}$. Using the maximum
principle in \eqref{h_1}, we have $\|h_{1}(t,.) \|_{L^{\infty}}
\leq C \|h_{1}(0,.) \|_{L^{\infty}}$ for $0 \leq t \leq T$. The maximum principle for
$|h_{2}|$ and $|h_{3}|$ is implemented in
similar fashion leading to estimates just like \ref{h_1}.

\end{proof}

\textbf{Proof of Theorem 2.} In Lemma 4 we have shown that the evolution of the relative entropy is
controlled by the term $\int_{0}^{T}r_{\epsilon}\, dt=\int_{0}^{T}r_{1,\epsilon}\, dt+
\int_{0}^{T}r_{2,\epsilon}\, dt$ which is bounded with the help of
\ref{In2} \& \ref{In3}. Using Proposition 2 and Lemma 5, it follows that
$\int_{0}^{T}r_{\epsilon}\, dt \to 0$ as $\epsilon \to 0$, proving Theorem 2. $\Box$

\subsection{Regularization of relative entropy}

The computations involving the relative entropy in this Section have been so far performed at a formal level,
i.e. by assuming smooth solutions with derivatives vanishing polynomially fast. It is not a hard task to
give a rigorous derivation of the results by performing a standard regularization argument which amounts to
regularizing all the
involved functions e.g. by convoluting with a mollifier, perform all the computations with the regularized ones, and
finally pass to the limit. In fact, the whole procedure we present here follows
closely the steps of the regularization argument in \cite{LeBrLi2}.

The regularization procedure will be presented here for the simpler case $H(\rho_{\epsilon}|\rho)$, since
there are less computation involved and the reader can get a better grasp of the full argument.
We begin with the assumption of smooth coefficients $G(x)$, $V(x)$ and we approximate a solution $f_{\epsilon}$ by
a mollified one $f_{\epsilon,\delta}=f_{\epsilon}\star \eta_{\delta}\in C^{\infty}(\mathbb{R}^{n}_{x}\times
\mathbb{R}^{n}_{v})$. The mollifier is $\eta_{\delta}=\frac{1}{\delta^{2n}}\eta\left(
\frac{x}{\delta},\frac{v}{\delta}\right)$, with
$\eta \in C^{\infty}_{c}(\mathbb{R}^{n}_{x}\times \mathbb{R}^{n}_{v})$ and $\iint \eta(x,v)\, dv dx=1$.

The equation for the regularized $f_{\epsilon,\delta}$ is (see \cite{LeBrLi2})
\begin{equation} \label{F-P-Reg}
\partial_{t}f_{\epsilon,\delta}+L_{\epsilon}f_{\epsilon,\delta}=U^{1}_{\epsilon,\delta}
+\nabla_{v}\cdot (G^{1/2}R^{1}_{\epsilon,\delta}) ,\end{equation}
where the expressions $U^{1}_{\epsilon,\delta}$, $R^{1}_{\epsilon,\delta}$ involve the following
commutators
\begin{align*} U^{1}_{\epsilon,\delta}=&-\frac{1}{\epsilon}[\eta_{\delta},v \cdot \nabla_{x}
-\nabla V(x)\cdot \nabla_{v}](f_{\epsilon})+\frac{1}{\epsilon^{2}}
[\eta_{\delta},(Gv)\cdot \nabla_{v}](f_{\epsilon})\\ &+
\frac{1}{\epsilon^{2}}[\eta_{\delta},\nabla_{v}\cdot (Gv)](f_{\epsilon})+\frac{1}{\epsilon^{2}}
[\eta_{\delta},G^{1/2}\nabla_{v}](G^{1/2}\nabla_{v}f_{\epsilon}), \\
R^{1}_{\epsilon,\delta}=&\frac{1}{\epsilon^{2}}[\eta_{\delta},G^{1/2}\nabla_{v}](f_{\epsilon}).\end{align*}
The notation we follow for commutators is
\begin{equation*} [\eta_{\delta},c](f)=\eta_{\delta}\star (cf)-c(\eta_{\delta} \star f) \quad
\text{and} \quad [\eta_{\delta},c_{1}](c_{2}f)=\eta_{\delta}\star
(c_{1}\cdot c_{2}f)-c_{1}\cdot(\eta_{\delta} \star c_{2}f),\end{equation*}
where $c$ is a differential operator (or vector), and $c_{1},c_{2}$ are general differential vectors.
This implies that the equation for $\rho_{\epsilon,\delta}:=\int f_{\epsilon,\delta}\, dv$ is
\begin{equation} \label{RhoE-Reg} \partial_{t}\rho_{\epsilon,\delta}+\frac{1}{\epsilon}\nabla_{x} \cdot
J_{\epsilon,\delta}=
\int U^{1}_{\epsilon,\delta}\, dv ,\end{equation}
where $J_{\epsilon,\delta}:=\int v f_{\epsilon,\delta}\, dv$.

The regularized limiting equation for $\rho$ (with $\rho_{\delta}=\rho \star \eta_{\delta}$) is
\begin{equation} \label{Rho-Reg} \partial_{t}\rho_{\delta}=\nabla_{x} \cdot (G^{-1}(\nabla_{x} \rho_{\delta}
+\nabla V(x)\rho_{\delta}))+U^{2}_{\delta}+\nabla_{x} \cdot (G^{-1/2}R^{2}_{\delta}) ,\end{equation}
with
\begin{align*} U^{2}_{\delta}=&[\eta_{\delta},\nabla_{x}\cdot (G^{-1}\nabla V(x))](\rho)+
[\eta_{\delta},G^{-1}\nabla V(x) \cdot \nabla_{x}](\rho) \\
&+[\eta_{\delta},\nabla_{x}\cdot G^{-1/2}](G^{-1/2}\nabla_{x}\rho)
+[\eta_{\delta},G^{-1/2}\nabla_{x}](G^{-1/2}\nabla_{x}\rho), \\
R^{2}_{\delta}=&[\eta_{\delta},G^{-1/2}\nabla_{x}](\rho) .\end{align*}

It has be shown (see Section 5.3 in \cite{LeBrLi2}) that
\begin{align*} U^{1}_{\epsilon,\delta},\, \, U^{2}_{\delta} &\xrightarrow{\delta \to 0} 0 \qquad
L^{\infty}+L^{2}([0,T],L^{1}_{loc})
\\ R^{1}_{\epsilon,\delta},\, \, R^{2}_{\delta} &\xrightarrow{\delta \to 0} 0 \qquad
L^{\infty}([0,T],L^{2}_{loc}),\end{align*}
for fixed $\epsilon >0$, as long as conditions in Proposition 1 are satisfied.

Next, multiplying \ref{F-P-Reg} by $v$ and integrating in velocity while
using the definition of $\rho_{\epsilon,\delta}$ we get
\begin{align} \nonumber \epsilon^{2} \partial_{t}J_{\epsilon,\delta}
&+ \epsilon(\nabla_{x}\rho_{\epsilon,\delta} +
\nabla V(x)\rho_{\epsilon,\delta}) +\epsilon \nabla_{x} \cdot \int \mathcal{M}
\nabla_{v}\left( \frac{f_{\epsilon,\delta}}{\mathcal{M}}\right) \otimes v \, dv \\ \label{J-Reg}
&+G J_{\epsilon,\delta}=\epsilon^{2}\int v \, U^{1}_{\epsilon,\delta}\, dv -
\epsilon^{2}\int G^{1/2}R^{1}_{\epsilon,\delta} \, dv .\end{align}

Since we want to take advantage of the fact that commutators vanish (as $\delta \to 0$) on compact sets,
we have to introduce a smooth cut-off function $\phi_{R}(x)=\phi\left( \frac{x}{R}\right)$, where $\phi$ is
a smooth function on $\mathbb{R}^{n}_{x}$, s.t. $0\leq \phi \leq 1$, with $\phi(x)=1$ for $|x|\leq 1$ and
$\phi(x)=0$ for $|x|\geq 2$. It follows that $\nabla \phi_{R}(\cdot)=
\frac{1}{R}\nabla \phi \left( \frac{\cdot}{R}\right)$. The idea is to include
the function $\phi_{R}$ in every integral and send $R \to \infty$,
after sending $\delta \to 0$. That way we can make integral terms that
involve commutators vanish.

For this reason, we introduce a relative entropy integral with a
cut-off $H_{R}(\rho_{\epsilon,\delta}|\rho_{\delta})=\int
\rho_{\epsilon,\delta} \log\frac{\rho_{\epsilon,\delta}}{\rho_{\delta}} \phi_{R}\, dx$.
Differentiating the entropy $H_{R}(\rho_{\epsilon,\delta}|\rho_{\delta})$ and using
\ref{RhoE-Reg}-\ref{J-Reg} we obtain
\begin{align*} \frac{d}{dt}H_{R}(\rho_{\epsilon,\delta}|\rho_{\delta})
&=\int \partial_{t}\rho_{\epsilon,\delta}(\log \rho_{\epsilon,\delta}+1)\phi_{R} \, dx -
\int \partial_{t}\rho_{\epsilon,\delta} \log \rho_{\delta} \phi_{R} \, dx \\ & -
\int \frac{\rho_{\epsilon,\delta}}{\rho_{\delta}} \partial_{t}\rho_{\delta}\phi_{R} \, dx= \ldots = I_{1}+I_{2}
+I_{3} .\end{align*}
The expressions $I_{1}$, $I_{2}$, $I_{3}$ that involve commutators are
\begin{align*} I_{1}(\epsilon,\delta,R)=\iint U^{1}_{\epsilon,\delta}(\log \rho_{\epsilon,\delta}+1)&\phi_{R}\, dv dx-
\iint U^{1}_{\epsilon,\delta}\log \rho_{\delta}\phi_{R}\, dv dx \\ &-\iint (U^{2}_{\delta}+\nabla_{x}\cdot
(G^{-1/2}R^{2}_{\delta}))\frac{\rho_{\epsilon,\delta}}{\rho_{\delta}}\phi_{R}\, dv dx ,\end{align*}
\begin{align*} I_{2}(\epsilon,\delta,R)=\frac{1}{\epsilon}\int J_{\epsilon,\delta} \cdot \nabla \phi_{R}(\log
&\rho_{\epsilon,\delta}+1)\, dx
-\frac{1}{\epsilon}\int J_{\epsilon,\delta}\cdot \nabla \phi_{R}\log \rho_{\delta}\, dx \\
&+\int \rho_{\epsilon,\delta}\nabla \phi_{R}\cdot G^{-1}
\left( \frac{\nabla_{x}\rho_{\delta}}{\rho_{\delta}}+\nabla V(x)\right)\, dx ,\end{align*}
and
\begin{equation*} I_{3}(\epsilon,\delta,R)=-\int  \Big| \frac{1}{\epsilon}G^{1/2}
\frac{J_{\epsilon,\delta}}{\rho_{\epsilon,\delta}}+G^{-1/2}
\left( \frac{\nabla \rho_{\delta}}{\rho_{\delta}}+\nabla V(x)\right)\Big|^{2}\rho_{\epsilon,\delta}\phi_{R}\,
dx+r'_{\epsilon,\delta,R} ,\end{equation*}
with the remainder term being
\begin{align*} r'_{\epsilon,\delta,R}&=\!\!-\!\!\int \!\!\! \left( \epsilon \partial_{t}J_{\epsilon,\delta}
+\!\!\nabla_{x}\!\cdot \!\int \mathcal{M}\nabla_{v}\left( \frac{f_{\epsilon,\delta}}{\mathcal{M}}\right) \otimes
v \, dv \!-\!\epsilon \!\! \int v U^{1}_{\epsilon,\delta}\, dv
+\!\epsilon \!\!\int G^{1/2}R^{1}_{\epsilon,\delta}\, dv \right) \\ & \qquad \qquad \cdot \left(
\frac{1}{\epsilon}\frac{J_{\epsilon,\delta}}{\rho_{\epsilon,\delta}}+
G^{-1}\left( \frac{\nabla \rho_{\delta}}{\rho_{\delta}}+\nabla V(x)\right)\right)\phi_{R} \, dx .\end{align*}

The trick is to take both $\delta \to 0$ and $R \to \infty$ while letting $\epsilon \to 0$. Since we have
the freedom of choice of how $\delta,R$ should behave for a fixed $\epsilon$, we will consider them as
functions of $\epsilon$ which we will describe in detail i.e. $\delta(\epsilon)$ and $R(\epsilon)$. Indeed,
for a given $\epsilon>0$, consider $\delta(\epsilon)$ s.t.
$|U^{1}_{\epsilon,\delta}|,|R^{1}_{\epsilon,\delta}| < \epsilon$, for all
$\delta < \delta(\epsilon)$. This way, we have $|U^{1}_{\epsilon,\delta}|,|R^{1}_{\epsilon,\delta}| \to 0$
while we let both $\epsilon$,\, $\delta(\epsilon)\to 0$.

If we consider $R$ fixed and take $\delta \to 0$, it is easy to see by the convergence
properties of commutators ($L^{\infty}$ in time) that
$\int_{0}^{T}I_{1}\, dt \to 0$. The exception is the last term of
$I_{1}$ that is treated separately. Same thing holds for the part of the remainder
term that involves commutators as we let $\delta \to 0$.

A bound for the first term in $I_{2}$ is
\begin{equation*} \Big| \frac{1}{\epsilon}\! \int_{0}^{T}\!\!\!\!\int \!\!
J_{\epsilon,\delta} \cdot \nabla \phi_{R}(\log
\rho_{\epsilon,\delta}+1)\, dx dt \Big| \leq \frac{1}{R} \| \nabla \phi\|_{L^{\infty}} \!\!
\int_{0}^{T}\!\!\!\!\int_{|x|>R}\!\!\! \frac{|J_{\epsilon,\delta}|}
{\epsilon} |(\log \rho_{\epsilon,\delta}+1)| \, dx dt.\end{equation*}
The exact same treatment holds for the second term in
$I_{2}$. It is obvious that these two integrals will vanish in
the limit $R \to \infty$ (partly due to the stability results
similar to Lemma 5). It will not matter how fast $R$ tends to infinity, so we can choose e.g.
$R(\epsilon)=\frac{1}{\epsilon}$.
For the third term in $I_{2}$, we have
\begin{align*} \Big| \int_{0}^{T}\!\!\!\!\int \rho_{\epsilon,\delta}\nabla \phi_{R}\cdot G^{-1}
\left( \frac{\nabla_{x}\rho_{\delta}}{\rho_{\delta}}+\nabla V(x)\right)\,& dx dt \Big| \leq
C  \Big \| \frac{G^{-1}(x)}{1+|x|}\Big\|_{L^{\infty}} \| \nabla \phi \|_{L^{\infty}} \\ & \cdot
\int_{0}^{T}\!\!\!\!\int_{|x|>R}
|\rho_{\epsilon,\delta}| \Big| \frac{\nabla_{x}\rho_{\delta}}{\rho_{\delta}}+\nabla V(x) \Big|\, dx dt,\end{align*}
which vanishes as $R \to \infty$ given the growth condition (see Proposition 1) in $G^{-1}(x)$ .

The last term in $I_{1}$ equals
\begin{equation*} \iint G^{-1/2}R^{2}_{\delta}\cdot
\nabla_{x}\left(\frac{\rho_{\epsilon,\delta}}{\rho_{\delta}}\right)\phi_{R}\,
dv dx +\iint G^{-1/2}R^{2}_{\delta}\cdot \nabla \phi_{R} \, \frac{\rho_{\epsilon,\delta}}{\rho_{\delta}} \, dv dx
\end{equation*}
and contains two terms. The first is treated like the terms in
$I_{1}$, and the second like these in $I_{2}$ with a
growth condition for $G^{-1/2}$ (Proposition 1).

In the last step, we send $\epsilon \to 0$ (while $\delta(\epsilon)\to 0$
and $R(\epsilon)\to \infty$) and combine this with the fact
that $\lim \limits_{\epsilon \to 0}\int_{0}^{T}I_{1}(\epsilon,\delta(\epsilon),R(\epsilon))\, dt=\lim \limits_{\epsilon
\to 0}\int_{0}^{T}I_{2}(\epsilon,\delta(\epsilon),R(\epsilon))\, dt=0$ to derive
\begin{equation*} \lim \limits_{\epsilon \to
0}H_{R(\epsilon)}(\rho_{\epsilon,\delta(\epsilon)}(T,\cdot)|\rho_{\delta(\epsilon)}(T)) \leq \lim \limits_{\epsilon \to
0}H_{R(\epsilon)}(\rho_{\epsilon,\delta(\epsilon)}(0,\cdot)|\rho_{\delta(\epsilon)}(0))  + \!\! \int_{0}^{T}\!\!\!\lim
\limits_{\epsilon \to 0} r'_{\epsilon,\delta,R}  \, dt. \end{equation*}
We finish with the estimates of previous subsection that
prove that the remainder term vanishes as $\epsilon \to 0$. This
yields the desired estimate
\begin{equation*} \lim \limits_{\epsilon \to 0}H(\rho_{\epsilon}(T,\cdot)|\rho(T,\cdot)) \leq
\lim \limits_{\epsilon \to 0}H(\rho_{\epsilon}(0,\cdot)|\rho(0,\cdot)) .\end{equation*}
Finally, we can remove the assumption on the smoothness
of coefficients by regularizing them in $x$ and
pass to the limit.

\begin{remark}
The regularization procedure above was carried out for the $H(\rho_{\epsilon}|\rho)$ relative entropy which
instantly implies $f_{\epsilon}\to \rho \mathcal{M}$ in $L^{1}$.
Indeed, by showing $L^{1}$ convergence of $\rho_{\epsilon}(t,x)$ to
the limiting distribution $\rho(t,x)$ it follows that $f_{\epsilon}$
converges to $\rho \mathcal{M}$ (in $L^{1}$) using the following simple argument.
We decompose $f_{\epsilon}-\rho \mathcal{M}$ as in
\begin{equation*} f_{\epsilon}-\rho \mathcal{M}=f_{\epsilon}-
\rho_{\epsilon}\mathcal{M}+(\rho_{\epsilon}-\rho)\mathcal{M} .\end{equation*}
It is trivial to show that the second term $(\rho_{\epsilon}-\rho)\mathcal{M}$ of the
decomposition $\to 0$ in $L^{1}$ by assumption.
For the first term $f_{\epsilon}-\rho_{\epsilon}\mathcal{M}$, we have
\begin{align*}
\| f_{\epsilon}-\rho_{\epsilon} \mathcal{M}\|_{L^{1}} & \leq \! \sqrt{2}
\left( \iint f_{\epsilon}\log{\frac{f_{\epsilon}}{\rho_{\epsilon}
\mathcal{M}}}\, dv \, dx \! \right)^{1/2} \!\!\! \leq \! \sqrt{2} \epsilon
\left( \iint \frac{|G^{-1/2}d_{\epsilon}|^{2}}{\epsilon^{2}}
\, dv \, dx \!\right)^{1/2}
\\ & \leq \sqrt{2} \epsilon C \to 0 \qquad \text{as} \quad \epsilon \to 0.
\end{align*}
The inequalities used in the first line are the Csisz\'ar-Kullback-Pinsker and log-Sobolev
in that order. Finally, the a priori energy bound (used in second line) concludes the argument.
\end{remark}

\subsection*{Acknowledgements.} The author is indebted to C. David Levermore and P-E Jabin for the
discussions that led to the birth of this work. Special
thanks to Athanasios Tzavaras for the initial motivation that led to the consideration
of this problem. Manoussos Grillakis and Julia Dobrosotskaya helped by proofreading this article.

DEPARTMENT OF MATHEMATICS, UNIVERSITY OF MARYLAND, COLLEGE PARK, U.S, 20742 \\ \\
E-MAIL: gmarkou@math.umd.edu

\end{document}